\documentclass[12pt,letterpaper,notitlepage]{amsart}
\usepackage{fullpage}

\usepackage{amssymb}
\usepackage{amsmath}
\usepackage{amsthm}
\usepackage{amsfonts}

\usepackage[linktocpage=true]{hyperref} 

\usepackage[numeric,msc-links,nobysame,lite]{amsrefs}
\usepackage{graphicx}
\usepackage{float}

\usepackage[active]{srcltx} 

\usepackage[applemac]{inputenc}

\newfont{\cmbsy}{cmbsy10}
\newfont{\cmmib}{cmmib10}

\newcommand{\Orden}{\mathop{\hbox{\cmbsy O}}\nolimits}

\DeclareMathOperator{\ph}{ph}
\DeclareMathOperator{\Arg}{Arg}
\DeclareMathOperator{\RH}{RH}

\def\Z{\mathbf{Z}}
\def\N{\mathbf{N}}

\def\R{\mathbf{R}}
\def\C{\mathbf{C}}

\def\Re{\mathrm{Re\,}}

\DeclareMathOperator*{\card}{card} 

\DeclareMathOperator{\angles}{\varphi}

\DeclareMathOperator*{\sgn}{sgn}
\def\Re{\operatorname{Re}}
\def\Im{\operatorname{Im}}
\def\medio{{\textstyle \frac{1}{2}}\,}
\def\cuarto{{\textstyle \frac{1}{4}}\,}
\def\mediot{{\textstyle \frac{t}{2}}\,}
\def\mediopi{{\textstyle \frac{\pi}{2}}}
\def\fracpi{{\textstyle \frac{1}{\pi}}}

\def\mediopit{{\textstyle \frac{\pi t}{2}}}
\def\terpi{{\textstyle \frac{3\pi}{2}}}

\def\SP{P}
\def\gg{\xi}

\newtheorem{theorem}{Theorem}
\newtheorem{definition}[theorem]{Definition}
\newtheorem{lemma}[theorem]{Lemma}
\newtheorem{proposition}[theorem]{Proposition}
\newtheorem{corollary}[theorem]{Corollary}

\newtheorem*{hypothesis}{Hypothesis \SP}

\theoremstyle{remark}
\newtheorem{remark}[theorem]{Remark}
\newtheorem{example}{Example}

\begin{document}

\title[On location of non-trivial zeros]
{On  the exact location of the non-trivial zeros of Riemann's zeta function}
\author[Arias de Reyna]{J. Arias de Reyna}
\address{Facultad de Matemáticas \\
Univ.~de Sevilla \\
Apdo.~1160
 \\
41080-Sevilla \\
Spain} 
\thanks{First author supported by  MINECO grant MTM2012-30748.}
\email{arias@us.es}

\author[van de Lune]{J. van de Lune}
\address{Langebuorren 49 \\
9074 CH Hallum \\
The Netherlands (formerly at the CWI, Amsterdam)} \email{j.vandelune@hccnet.nl}

\date{\today}

\begin{abstract}
In this paper we introduce the real valued real analytic function  $\kappa(t)$  
implicitly defined by 
\[e^{2\pi i \kappa(t)}=-e^{-2i\vartheta(t)}\frac{\zeta'(\frac12-it)}{
\zeta'(\frac12+it)},\qquad  (\kappa(0)=-\medio).\] 
By studying the equation  $\kappa(t) = n$  (without making any unproved hypotheses),
we will show that (and how) this function is closely related to the (exact) position of the zeros of  Riemann's   $\zeta(s)$ and  $\zeta'(s)$.
Assuming the Riemann hypothesis and the simplicity of the zeros of  $\zeta(s)$, 
it will follow that the ordinate of the zero
$1/2 + i \gamma_n$  of  $\zeta(s)$  will be the unique solution to the equation  $\kappa(t) = n$.
\end{abstract}

\maketitle

\tableofcontents

\section{Introduction.}

The functional equation of Riemann's zeta function $\zeta(s)$  implies that
$\zeta(\medio+it)=Z(t)e^{-i\vartheta(t)}$
where $Z(t)$ and $\vartheta(t)$ are real valued and real analytic 
functions and the \emph{phase}
$-\vartheta(t)$ is a rather simple function depending only on Euler's 
gamma function $\Gamma(s)$. 
An analogous decomposition is valid for any meromorphic function. We give a 
formal definition of the phase of a  real analytic function in Section
\ref{phasearg}. 

We will
define some functions related to the zeros of $\zeta(s)$ and the phase of related
functions.
Of course, these functions have appeared in the literature but only in an 
implicit way and have not been studied for their own sake.  For example,
Levinson and Montgomery \cite{LM} define 
\[J(\medio+it):=\zeta(\medio+it)+\zeta'(\medio+it)\Bigl[
\frac{h'(\frac12+it)}{h(\frac12+it)}+\frac{h'(\frac12-it)}{h(\frac12-it)}\Bigr]^{-1}\]
where $h(s)=\pi^{-s/2}\Gamma(s/2)$, 
and assert that \emph{the determination of the number of zeros of $\zeta(s)$ in 
$[0< ]\sigma<\frac12$ can be conveniently ascertained from the variation of 
$\arg J(\frac12+it)$}.

They do not use the simplified form 
\[J(\medio+it)=-e^{-2i\vartheta(t)}\frac{\zeta'(\medio-it)}{2\vartheta'(t)}.\]
With our notations we would have
\[\ph J(\medio+it)=\pi-2\vartheta(t)-\ph\zeta'(\medio+it)=\mediopi+\pi\kappa(t)-\vartheta(t)=2\pi-E(t).\]
Here $\kappa(t)$ is the main function introduced here. It is closely connected with the zeros of $\zeta(s)$, and is
%For example, if, to simplify things, we assume the Riemann hypothesis (RH for short)
%and the simplicity of its zeros, we will have $\kappa(\gamma_n)=n$. Here as  usual 
%we denote by $\gamma_n$ the ordinate of the $n$-th zero of zeta in the upper half-plane.
%The function $\kappa(t)$ is also 
implicitly used in Levinson \cite{L}*{equation (1.6)}  to prove that more than 
$1/3$ of the zeros of $\zeta(s)$ are on the critical line.

In our paper we seldomly assume the RH, and use the standard notations of
the subject. Therefore we shall denote  the zeros of $\zeta(s)$ on the upper 
half-plane by
$\beta_n+i\gamma_n$ (where $\beta_n$ and $\gamma_n$ are real numbers) and
$0<\gamma_1\le\gamma_2\le\cdots$. If a zero is multiple with multiplicity $m$, then 
it appears precisely $m$ times consecutively in the above sequence. 
\cite{T}*{Chapter 9, p.~214}.  We shall need to introduce another related sequence of \emph{real}
numbers $(0<)\ \gg_1<\gg_2<\cdots $ defined so that the set
$\{\gg_n\colon n\in\N\}:=\{ t>0:  \zeta(\medio+it)=0\}$. Here only the ordinates of 
the zeros  on the critical line appear. These $\gg_n$ do not repeat by any 
circumstance.

The two sequences $(\gg_n)$ and $(\gamma_n)$ coincide if and only if the RH is 
true and all the zeros of $\zeta(s)$ on the critical line are simple.

Even in  case the RH were not true,  
we will show that $\kappa(t)$ is related
to the zeros of $\zeta(s)$ on the critical line. We will show that $\kappa(\gg_n)=n$ for all natural numbers, independently of any hypothesis.

The relations between the zeros of $\zeta(s)$ and $\zeta'(s)$ has been the object
of much study. Starting with Speiser \cite{S} who showed that the RH is equivalent 
to $\zeta'(s)$ having no zeros in $0<\sigma<\frac12$, Levinson and Montgomery \cite{LM}
give a quantified version of Speiser's theorem, and Berndt \cite{B} gives an estimation of 
the number of zeros of $\zeta'(s)$ to a given height. Great interest in the zeros of $\zeta'(s)$ 
is related to their horizontal distribution, in which many questions remain open
(see Levinson Montgomery \cite{LM},  Conrey and Ghosh \cite{CG}, Soundararajan 
\cite{So}, Zhang \cite{Z}, Garaev and Y{\i}ld{\i}r{\i}m \cite{GY}).  Here we get 
a new way to study these relationships by means of our function $\kappa(t)$.  The number of zeros of $\zeta(s)$ on an interval of the 
critical line not counting multiplicities is related to the increment of $\kappa(t)$
in this interval. 
Assuming the RH this function will be strictly increasing, so $\kappa'(t)\ge0$.  The connection
is by means of equation \eqref{E:kappaprime} which represents this function in terms of the zeros of  $\zeta'(s)$.

Therefore, $\kappa(t)$ is related to the zeros of $\zeta(s)$ (Prop.~\ref{P:17}), 
and $\kappa'(t)$ is fully
determined by the zeros of $\zeta'(s)$ (Prop.~\ref{P: kappaprime}).  
The relationship of $\kappa'(t)$ with the
zeros of $\zeta(s)$ is also direct and double (Prop.~\ref{kappaprimeatzeros} and 
equation \eqref{E:45}). See figure \ref{F:two ways} for a 
graphical description of these relations. 

In Section \ref{phasearg} we give the definition and (some simple) properties of the 
decomposition in \emph{phase} and \emph{signed modulus} of a real analytic function. In particular,
in Proposition \ref{phaseint} we will write the phase as a convergent integral. 
After this 
we devote Section \ref{S:theta} to  some properties of the phase 
$-\vartheta(t)$ of 
$\zeta(\medio+it)$. Since we will use its convexity for all $t>0$, we give a
simple derivation of this fact. Section \ref{S:kappa} is devoted to the 
introduction of $\kappa(t)$.  The definition in Proposition \ref{E:defk} 
\[e^{2\pi i
\kappa(t)}=1+2\vartheta'(t)\frac{\zeta(\medio+it)}{\zeta'(\medio+it)},\qquad
\kappa(0)=-\medio.\]
is possible because the function in the right hand side makes a circular movement for 
$t\in\R$. We study the relationship of $\kappa(t)$ with $\ph\zeta'(\medio+it)$ and
$\vartheta(t)$.   The function $\kappa(t)$  is a complicated function, its behavior 
being
connected with the RH. We show here the equation
\[\kappa(\gg_n)=n\]
which determines the set of \emph{real} numbers $t$ with $\zeta(\medio+it)=0$. 

Proposition \ref{midvalues} may come as a  surprise. It relates the points 
where  $\kappa(t)$ is 
half an integer with the zeros of $Z'(t)$. Assuming the RH the function $\kappa(t)$ will 
be strictly increasing and between $\gamma_n$ and $\gamma_{n+1}$ there would be
only one zero of $Z'(t)$, situated just at the point where $\kappa(t)=n+\frac12$.  
In the next
section we show what of this remains true if we do not assume the RH, and see the 
first application of the function $\kappa'(t)$.

The main result of Section \ref{S:conection} is a formula for $\kappa'(t)$
in terms of the zeros of $\zeta'(s)$ (see Proposition \ref{P: kappaprime}).  
Therein appears a constant $A$ which we relate in equation \eqref{E:A} with the zeros
of $\zeta'(s)$.  In Section \ref{counting} we  obtain the value $A=\frac12\log 2$. 
We give a  proof that relates this constant to the difference in the counting
of zeros of $\zeta(s)$ given by Riemann and the one for the zeros of $\zeta'(s)$
given by Berndt.  Also we include a proof that the RH implies $\kappa'(t)>0$ for 
$t>a_\kappa$. 

Section \ref{kappaprimeandzeros} establishes the connection of $\kappa'(t)$ with the 
zeros of $\zeta(s)$. We know from Section \ref{S:kappa} that for  $n<m$ we have
$\int_{\gg_n}^{\gg_m}\kappa'(t)\,dt = m-n$.  We show that 
$\kappa'(\gg_n)=\vartheta'(\gg_n)/\omega$ where $\omega$
is the multiplicity of the zero $\frac12+i\gg_n$ of $\zeta(s)$.  In Proposition 
\ref{T:GY} we apply 
these relationships to give,
assuming the RH,  a new proof of a strengthening of a  Theorem of 
Garaev and Y{\i}ld{\i}r{\i}m \cite{GY} (which they prove unconditionally). 
In Section \ref{S:E-S} we introduce a related function $E(t)$ and show its
relationship with the classical function $S(t)$ and with a function $\RH(t)$ which counts the failures up to height $t$ of both the RH and 
the simplicity of the zeros of $\zeta(s)$. This is almost the function considered
by Levinson and Montgomery. 

Most of the functions appearing in this paper were found some years ago (in 1997)
 by one of us 
(JvdL) while searching  for a formula (or equation) for the exact
location of the non-trivial zeros of the Riemann zeta function.

\section{Phase and argument of a  function.}\label{phasearg}

The results in this section are easy but we did not find any proper
references. We include the simple proofs and introduce our 
notations about \emph{phase} and \emph{argument} of a real analytic function. 

\begin{definition}
A function  $f\colon\R\to\C$ is called real analytic if for every
$t_0\in\R$ there exists a convergent power series $P(z)=\sum_{k=0}^\infty
c_kz^k$ such that $f(t)=P(t-t_0)$ for all  $t$ in a neighborhood of
$t_0$. In other words: A function   $f\colon\R\to\C$ called real analytic 
if  $f$  has an analytic extension to a neighborhood of  $\R$.
\end{definition}

\begin{proposition}\label{P:arg}
If  $f\colon \R\to\C\smallsetminus \{0\}$ is a real analytic
function, then there exists a real analytic function $g$ such that
$f(t)=e^{g(t)}$ for every $t\in\R$.
\end{proposition}

\begin{proof}
For every $t_0\in\R$  let $\Delta(t_0)$ be a disk with center at
$t_0$ such that $f(t)=P(t-t_0)$ for $t\in\Delta(t_0)\cap\R$, and
such that $P(z-t_0)\ne0$ for $z\in\Delta(t_0)$. The union
$G=\bigcup_{t_0} \Delta(t_0)$ is a simply connected domain and $f$
can be extended to $G$ as an analytic  function. Since $f(z)\ne 0$
for $z\in G$, there exists an analytic function $g$ on $G$ such that
$f(z)=e^{g(z)}$ for all $z\in G$.
\end{proof}

\begin{corollary}\label{P:argcor}
If  $f\colon \R\to\C\smallsetminus\{0\}$ is a real analytic
function, then there exists a real analytic function
$\varphi\colon\R\to\R$ such that $f(t)=|f(t)|e^{i\varphi(t)}$.
\end{corollary}

We write in such a  case $\varphi(t)=\arg f(t)$.   This is an
analytic (and hence continuous) determination of the argument of $f$.  
Two such functions
differ only by  an integral multiple of $2\pi$.

\begin{proposition}\label{P:ph}
If $f\colon \R\to\C$ is a real analytic function, then there are two
real  analytic functions $U\colon\R\to\R$ and
$\varphi\colon\R\to\R$ such that
\begin{displaymath}
f(t)=U(t)e^{i\varphi(t)}.
\end{displaymath}
Given two such representations, $f= U_1e^{i\varphi_1}$ and $f=U_2
e^{i\varphi_2}$, we have either $U_1=U_2$ and
$\varphi_1-\varphi_2=2k\pi$ or $U_1=-U_2$ and
$\varphi_1-\varphi_2=(2k+1)\pi$  for some integer $k$.
\end{proposition}

\begin{proof}
If $f$ does not vanish, then $|f|$ is real analytic  and by
Corollary \ref{P:argcor} there exists a real analytic  function
$\varphi\colon\R\to\R$ such that $f\;|f|^{-1}=e^{i\varphi}$, and
we can take $U=|f|$ in this case.

Now assume that $f$ has real zeros.  Let $a_n$ be the real zeros of
$f(t)$ listed with multiplicities. We may assume that $a_1=\cdots
=a_m=0$ and all the others non-zero. By Weierstrass' factorization
theorem there exists an entire function
\begin{displaymath}
g(z)=z^m\prod_{n>m} E_{n-1}(z/a_n),\qquad  z\in\C
\end{displaymath}
whose zeros are the numbers $a_n$, and the $E_n(z)=(1-z)e^{z+z^2/2+\cdots +z^n/n}$ are 
the canonical factors.  Observe also that this function
is real for real $z=t$. By the previous argument there exist real
analytic functions $h$ and $\varphi$ such that  $f/g= h
e^{i\varphi}$. Thus $f=(g h) e^{i\varphi}$, and $U=gh$. This
proves that the claimed decomposition exists.

Finally, if $f=U_1 e^{i\varphi_1}=U_2 e^{i\varphi_2}$, then
$U_1/U_2$ is a real analytic function without zeros. Also
$|U_1/U_2|=|e^{i(\varphi_2-\varphi_1)}|=1$ and it follows that
$U_1/U_2$ is either equal to $1$ or to $-1$. In the first case
$e^{i(\varphi_2-\varphi_1)}=1$ and
$\varphi_2=2k\pi+\varphi_1$ for some integer $k$. The other case
may be treated similarly.
\end{proof}

\begin{definition}
Given a real analytic function $f\colon\R\to\C$ we call phase of $f$
any real analytic  function $\ph f\colon\R\to\R$ such that
$f(t)=U(t)e^{i\ph f(t)}$ with $U\colon\R\to\R$ a real analytic
function.
\end{definition}

If $g_1$ and $g_2$ are two such functions there exists an integer
$k$ such that $g_1(t)=g_2(t)+k\pi$ for every $t\in\R$.  

Observe that the above definition is not standard. We are 
making use of the 
word \emph{phase} with a peculiar mathematical meaning.

The main difference
between the phase  of a real analytic function and its argument is that
for some $t\in\R$ the value $\ph f(t)$ may not be equal to one of the arguments
of the complex number $f(t)$. We will have only $\ph f(t) $ equal to this 
argument modulo $\pi$.  

\begin{example}
It is easy to check that 
\begin{equation}\label{cosfase}
\cos\mediopi(\medio+it)=\frac{1}{\sqrt{2}}\sqrt{\cosh\pi
t}\,e^{-i\arctan(\tanh\frac{\pi t}{2})}.
\end{equation}
\end{example}
\begin{example}\label{exampleone}
One of the most interesting examples is that of the zeta function on 
the critical line. In this case we have  (see Edwards \cite{E}*{p.~119})
\begin{equation}\label{zetafase}
\zeta(\medio+it)=Z(t)e^{-i\vartheta(t)}
\end{equation}
where $Z\colon\R\to\R$ and $\vartheta\colon\R\to\R$  are real analytic. $Z(t)$ is the Riemann-Siegel function
(sometimes called Hardy function \cite{I}).
\end{example}
\begin{example}\label{examplewo} 
The phase  $-\vartheta(t)$ in Example \ref{exampleone} is 
related to the phase of $\Gamma(\frac14+i\frac{t}{2})$ by 
\begin{equation}
\Gamma(\cuarto+i\;\medio t)=\left|\Gamma(\cuarto+i\;\medio
t)\right|\; e^{i(\vartheta(t)+\mediot \log\pi)}.
\end{equation}
(For more details see \cite{T}*{(4.17.2)}).
\end{example}
\begin{example} We have not found any reference for our next example:
\begin{equation}\label{gammafase2}
\Gamma(1/2+it)=\sqrt{\frac{\pi}{\cosh\pi
t}}\exp\bigl\{i(2\vartheta(t)+t\log(2\pi)+\arctan\tanh\mediopit)\bigr\}.
\end{equation}
This may be shown using only properties of  $\Gamma(s)$ but we present  a proof
based on the functional equation of $\zeta(s)$. 

Let $\Phi(s)=\medio\zeta(s)\zeta(1-s)$. Then, 
by the functional equation
\begin{equation}\label{phis}
\Phi(s)=\cos\frac{\pi
s}{2}(2\pi)^{-s}\Gamma(s)\zeta(s)^2.
\end{equation}
Substituting  \eqref{cosfase} and \eqref{zetafase} into this equation, 
we get with $s=\frac12+it$
\begin{equation}
\medio Z(t)^2=\sqrt{\frac{\cosh(\pi t)}{2}}\;e^{-i\arctan(\tanh\frac{\pi
t}{2})}\cdot
\frac{1}{\sqrt{2\pi}}e^{-it\log(2\pi)}\cdot\Gamma(s)\cdot
e^{-2i\vartheta(t)}Z(t)^2
\end{equation}
from which \eqref{gammafase2} follows for $Z(t)\ne0$. But since the argument in 
\eqref{gammafase2} is real analytic it is true for all $t$.
\end{example}

\begin{proposition}\label{phaseint}
If  $f$ is a non-constant real analytic function, then for every
$t\in\R$ we have
\begin{equation}\label{E:phase}
\ph f(t)=\ph f(0)+\int_0^t\Im\frac{f'(x)}{f(x)}\,dx.
\end{equation}
\end{proposition}

\begin{proof}
The function $\ph f(t)$ is real analytic, so that
\begin{displaymath}
\ph f(t)=\ph f(0)+\int_0^t(\ph f)'(x)\,dx.
\end{displaymath}
There exists a real analytic function $U$ such that
$f(t)=U(t) e^{i\ph f(t)}$. Therefore, if $f(x)\ne0$ then
\begin{displaymath}
\frac{f'(x)}{f(x)}=\frac{U'(x)}{U(x)}+ i (\ph f)'(x)
\end{displaymath}
so that
\begin{displaymath}
(\ph f)'(x)=\Im\frac{f'(x)}{f(x)}.
\end{displaymath}
It follows that $\Im\frac{f'(x)}{f(x)}$ is in fact a real analytic
function, the possible singularities at the points where $f(x)=0$
being removable.
\end{proof}

\begin{example}\label{exampletheta}
By Examples \ref{exampleone} and \ref{examplewo}  we  have
\begin{equation}\label{intvartheta}
\vartheta(t)=-\int_0^t\Re\frac{\zeta'(\frac12+ix)}{\zeta(\frac12+ix)}\,dx=-\frac{t}{2}\log\pi+\frac{1}{2}
\int_0^t\Re\frac{\Gamma'(\frac14+i\frac{x}{2})}
{\Gamma(\frac14+i\frac{x}{2})}\,dx.
\end{equation}
\end{example}

\section{The function $\vartheta(t)$.}\label{S:theta}

In this section we recall some properties of the function $\vartheta(t)$ introduced in Example \ref{exampleone}.

We need to show that $\vartheta(t)=0$ has only one solution for $t>0$. To this
end we must show 
explicit formulae for $\vartheta(t)$ for \emph{small} $t$, which are seldomly
considered.

Indeed, after having introduced $\vartheta(t)$, most authors
immediately start discussing its asymptotic expansion (compare
Edwards \cite{E}*{p.~119} and Gabcke \cite{G}*{p.~4}).

\begin{proposition}
For $\vartheta(t)$ we have the following series expansion
($\,$con\-vergent for all $t\in\R$$\,$) ($\,$$\gamma$ being Euler's
constant$\,$)
\begin{equation}\label{E:thetaTotal}
\vartheta(t)=-\medio\bigl(\gamma+\log\pi+3\log2+\mediopi\bigr)t+
\sum_{k=0}^\infty\Bigl(\frac{2t}{4k+1}-\arctan\frac{2t}{4k+1}\Bigr).
\end{equation}
\end{proposition}

\begin{proof} From the Weierstrass product for $\Gamma(s)$ we obtain
\begin{multline*}
-\Im\log\Gamma\bigl(\cuarto+i\,\mediot)=\\
=\arg\Bigl\{
(\cuarto+i\,\mediot)e^{\gamma(\cuarto+i\,\mediot)}\prod_{k=1}^\infty
\Bigl(1+\frac{\cuarto+i\,\mediot}{k}\Bigr)e^{-(\cuarto+i\,\mediot)/k}\Bigr\}+2\ell\pi\\
=\arctan
(2t)+\frac{\gamma}{2}\,t+\sum_{k=1}^\infty\Bigl(\arctan\frac{2t}{4k+1}-\frac{t}{2k}\Bigr)-2k\pi
\end{multline*}
for some $\ell$,  $k\in\Z$.

Since  $\vartheta(t)=\arg\Gamma(\frac14+i\frac{t}{2})-\frac{t}{2}\log\pi$
we have
\begin{displaymath}
\vartheta(t)= -\frac{\gamma+\log\pi}{2}\,t-\arctan
(2t)-\sum_{k=1}^\infty\Bigl(\arctan\frac{2t}{4k+1}-\frac{t}{2k}\Bigr)+2k\pi
\end{displaymath}
so that, taking $t=0$, we find that $k=0$. We rewrite the last
series as follows
\begin{displaymath}
\sum_{k=1}^\infty\Bigl(\arctan\frac{2t}{4k+1}-\frac{t}{2k}\Bigr)
=-\frac{t}{2} \sum_{k=1}^\infty
\Bigl(\frac{1}{k}-\frac{1}{k+\cuarto}\Bigr)+
\sum_{k=1}^\infty\Bigl(\arctan\frac{2t}{4k+1}-\frac{2t}{4k+1}\Bigr).
\end{displaymath}
The first series  in the right hand side can be summed explicitly
\begin{multline*}
\sum_{k=1}^\infty
\Bigl(\frac{1}{k}-\frac{1}{k+\cuarto}\Bigr)=\sum_{k=1}^\infty\Bigl(\int_0^1u^{k-1}\,du-\int_0^1
u^{k-1+\frac14}\,du\Bigr)
= \int_0^1\frac{1-u^{1/4}}{1-u}\,du= \\
=4\int_0^1\frac{1-v}{1-v^4}\,v^3\,dv
=\int_0^1\frac{4v^3\,dv}{(1+v)(1+v^2)} 
=4-3\log2-\frac{\pi}{2}.
\end{multline*}
Combining these equations we obtain \eqref{E:thetaTotal}.
\end{proof}

\begin{proposition}
For every $t\in\R$ we have
\begin{equation}\label{E:thetaprime}
\vartheta'(t)=-\medio(\gamma+\log\pi)-\frac{2}{1+4t^2}-\sum_{k=1}^\infty
\Bigl(\frac{2(4k+1)}{(4k+1)^2+4t^2}-\frac{1}{2k}\Bigr).
\end{equation}
\end{proposition}

\begin{corollary}\label{convexity}
The function $\vartheta(t)$ is convex on $(0,+\infty)$, and there exists
a unique positive real number $a_\vartheta$ where
$\vartheta'(a_{\vartheta})=0$.
\end{corollary}

By differentiation of \eqref{E:thetaprime} we get
\begin{equation}
\vartheta''(t)=16t\sum_{k=0}^\infty \frac{4k+1}{((4k+1)^2+4t^2)^2}
\end{equation}
from which the Corollary follows.

This Corollary is proved in \cite{Lm}*{Lemma 11, Lemma 12}. We have
\[a_\vartheta= 
6.28983\,59888\,36902\,77966\,50901\,00821\,85339\,66583\,12945\,19278\,95453\,
57765\,\dots\]

\section{The function $\kappa(t)$.}\label{S:kappa}

The next Proposition is included in  Titchmarsh
\cite{T}*{p.~291}, but we write the proof below, because we 
are also interested in the  formulas used. 

\begin{proposition}\label{P:doblezeros}
If $\zeta'(\medio+ia)=0$ for a real $a$, then
$\zeta(\medio+ia)=0$.
\end{proposition}

\begin{proof}
We start from $\zeta(\medio+it)=e^{-i\vartheta(t)}Z(t)$.
Differentiation  with respect to $t$ yields
\begin{displaymath}
i\zeta'(\medio+it)=-i\vartheta'(t)e^{-i\vartheta(t)}Z(t)+e^{-i\vartheta(t)}Z'(t).
\end{displaymath}
Multiplying this by $-ie^{i\vartheta(t)}$ we get
\begin{equation}\label{zetavartheta}
e^{i\vartheta(t)}\zeta'(\medio+it)=-\vartheta'(t)Z(t)-iZ'(t)
\end{equation}
and taking  real parts we obtain 
\begin{equation}
-\vartheta'(t)Z(t)=\Re\bigl\{e^{i\vartheta(t)}\zeta'\bigl(\medio+it\bigr)\bigr\}
\end{equation}
which may also be written as
\begin{equation}\label{E:dos}
-2\vartheta'(t)Z(t)=e^{i\vartheta(t)}\zeta'\bigl(\medio+it\bigr)+
e^{-i\vartheta(t)}\zeta'\bigl(\medio-it\bigr).
\end{equation}
Let us assume that $\zeta'\bigl(\medio+it\bigr)=0$ for some real $t$. Since
$\zeta'(\frac12-it)=0$ we may assume that $t>0$ and   we get $\vartheta'(t)Z(t)=0$.   Since $\vartheta'(t)=0$ only for
$t=a_\vartheta\approx 6.29$ where $\zeta'(\frac12+ia_\vartheta)\ne0$, we get $Z(t)=0$.
Therefore $\zeta'\bigl(\medio+it\bigr)=0$ implies $\zeta\bigl(\medio+it\bigr)=0$
\end{proof}

Recall that we denote, as usual, by $\beta_n+i\gamma_n$ the non-trivial zeros of $\zeta(s)$, ordered in such a way that 
$(0<)\ \gamma_1\le\gamma_2\le\cdots $, repeating each term according to its multiplicity.  We will need another
related sequence. 
Let $(0<)\ \gg_1<\gg_2<\cdots $ be the sequence of \emph{real} numbers $t$
such that $\zeta(\medio+it)=0$, counted without multiplicities. Hence the $\gg_n$ 
\emph{only 
denote zeros on the critical line}.  If we assume the RH and the simplicity of the zeros, we
would, of course,  have $\gg_n=\gamma_n$.

\begin{proposition}\label{P:uno}
For every real $t\ne\pm\gg_n$ we have
\begin{equation}\label{eq:uno}
1+2\vartheta'(t)\frac{\zeta(\medio+it)}{\zeta'(\medio+it)}=-e^{-2i\vartheta(t)}
\frac{\zeta'(\medio-it)}{\zeta'(\medio+it)}.
\end{equation}
\end{proposition}

\begin{proof}

Multiplying \eqref{E:dos} by $e^{-i\vartheta(t)}$ we get
\begin{displaymath}
-2\vartheta'(t)\zeta(\medio+it)=\zeta'\bigl(\medio+it\bigr)+
e^{-2i\vartheta(t)}\zeta'\bigl(\medio-it\bigr).
\end{displaymath}
Since $t\ne \pm\gg_n$, and using Proposition \ref{P:doblezeros} we have 
$\zeta'(\medio+it)\ne0$, so that we can
divide by $\zeta'(\medio+it)$ and obtain our result.
\end{proof}

\begin{proposition}
There exists a unique real analytic function $\kappa\colon\R\to\R$
such that
\begin{equation}\label{E:defk}
e^{2\pi i
\kappa(t)}=1+2\vartheta'(t)\frac{\zeta(\medio+it)}{\zeta'(\medio+it)},\qquad
\kappa(0)=-\medio.
\end{equation}
\end{proposition}

\begin{proof}
By Proposition \ref{P:uno} the function $f\colon\R\to\C$ defined by
\begin{displaymath}
f(t)=1+2\vartheta'(t)\frac{\zeta(\medio+it)}{\zeta'(\medio+it)}
\end{displaymath}
satisfies $|f(t)|=1$ for $t\ne \gg_n$. By definition, and
Proposition \ref{P:doblezeros}, $f$ is real analytic and satisfies
$|f(\gg)|=1$, so that there exists  a real analytic
$\kappa\colon\R\to\R$ such that $f(t)=e^{2\pi i \kappa(t)}$. This
function is uniquely defined by its value at any one point. Since
$\vartheta(0)=0$  we have $f(0)=-1$ (see \eqref{eq:uno}) and we can take
$\kappa(0)=-\medio$.
\end{proof}

Applying Proposition \ref{P:ph} to $\zeta'(\medio+it)$ we arrive at
two real analytic functions $\rho\colon\R\to\R$ and
$\ph\zeta'(\medio+it)$. Observing that $\zeta'(\medio)<0$ we may choose 
\begin{displaymath}
\zeta'(\medio+it)=\rho(t)e^{i\ph\zeta'(\frac12+it)}, \quad
\rho(0)=|\zeta'(\medio)|,\quad \ph\zeta'(\medio)=\pi.
\end{displaymath}
If we assume that $\zeta(s)$ has no multiple zero  on the
critical line, then  $\zeta'(\medio+it)\ne0$ and we will have
$\rho(t)=|\zeta'(\medio+it)|$ and
$\ph\zeta'(\medio+it)=\arg\zeta'(\medio+it)$
(where $\arg\zeta'(\medio+it)$ is meant to be a continuous function of $t$
in $\R$).

\begin{figure}[H]
  \includegraphics[width=\hsize]{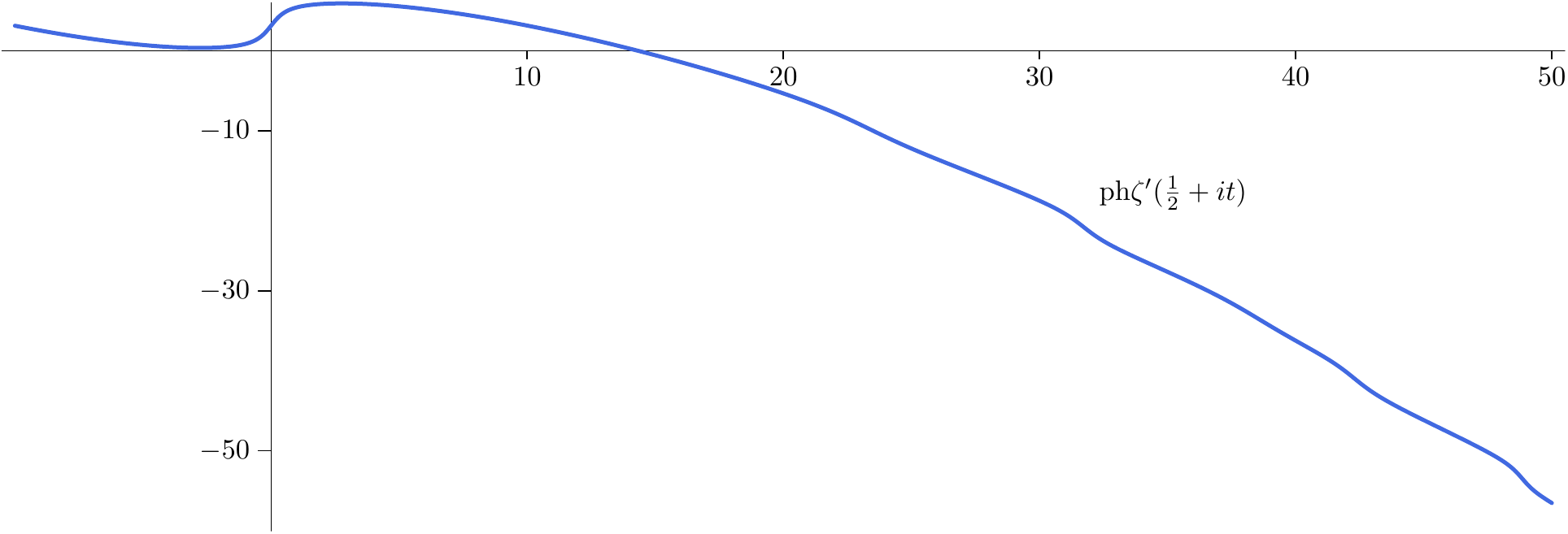}
  \caption{$\ph\zeta'(\medio+it)$}
  \label{F:phasezetap2}
\end{figure}

\begin{proposition} For all $t\in\R$
\begin{equation}\label{E:kappa1}
\kappa(t)=\frac{1}{2}-\frac{1}{\pi}\bigl(\vartheta(t)+\ph
\zeta'(\medio+it)\bigr).
\end{equation}
\end{proposition}

\begin{proof}
By the definition of $\kappa(t)$ and \eqref{eq:uno} we have
\begin{displaymath}
\exp\bigl(2\pi i
\kappa(t)\bigr)=-e^{-2i\vartheta(t)}\frac{\zeta'(\medio-it)}{\zeta'(\medio+it)}
=\exp\bigl(\pi i-2i\vartheta(t)-2i\ph\zeta'(\medio+it)\bigr).
\end{displaymath}
Hence there exists an integer $n$ such that
\begin{displaymath}
2\pi i \kappa(t)=\pi i-2i\vartheta(t)-2i\ph\zeta'(\medio+it)+2\pi i
n.
\end{displaymath}
For $t=0$ we get $n=0$ and \eqref{E:kappa1} follows.
\end{proof}

\begin{corollary} For every real $t$ we have
\begin{equation}\label{E:kappa}
\kappa(t)=-\frac{1}{2}-\frac{1}{\pi}\int_0^t\Bigl(\vartheta'(x)+
\Re\frac{\zeta''(\medio+ix)}{\zeta'(\medio+ix)}\Bigr)\,dx.
\end{equation}
\end{corollary}

\begin{proof}
In formula \eqref{E:kappa1}, we replace $\ph\zeta'(\medio+it)$
by the integral expression given by \eqref{E:phase}.
\end{proof}

Observing that $-\vartheta(t)$ is the phase of $\zeta(\frac12+it)$ we  also
obtain  (see \eqref{intvartheta})
\begin{equation}\label{kappanature}
\kappa(t)=-\frac{1}{2}+\frac{1}{\pi}\int_0^t\Re\Bigl(\frac{\zeta'(\frac12+ix)}{
\zeta(\frac12+ix)}
-\frac{\zeta''(\frac12+ix)}{\zeta'(\frac12+ix)}\Bigr)\,dx.
\end{equation}

\begin{figure}[H]
  \includegraphics[width=\hsize]{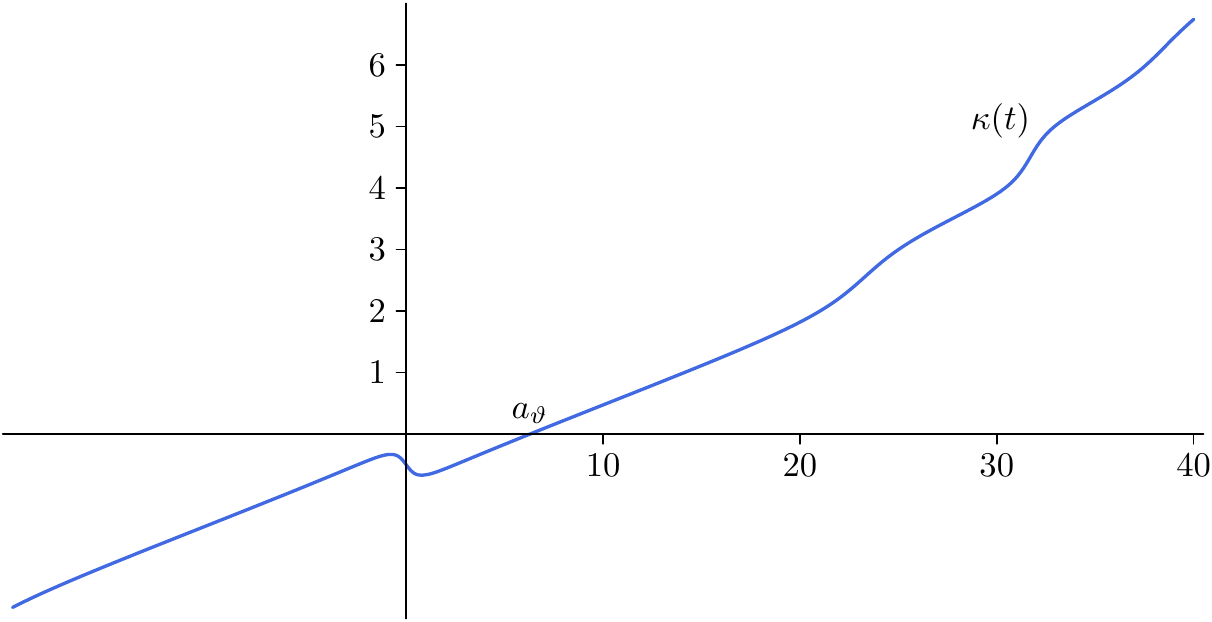}
  \caption{$\kappa(t)$}
  \label{F:F1}
\end{figure}
From \eqref{kappanature} we see that $\kappa(t)+\frac12$ is an odd function.

%\begin{proposition}
%The functions $\vartheta(t)$, $\kappa(t)+\frac12$, $\ph\zeta(\frac12+it)-\pi$ are odd,
%$Z(t)$ is even.
%\end{proposition}
%
%\begin{proof}
%It is known that $\vartheta(t)$ is odd, it follows easily from equation 
%\eqref{intvartheta} or
%\eqref{E:thetaTotal}. Being $\vartheta(t)$ an odd function and 
%$\zeta(s)$ real for $s$ real 
%implies by equation \ref{zetafase} that $Z(t)$ is even.  
%
%That $\kappa(t)+\frac12$ is odd follows from equation \eqref{E:kappa}, 
%because the integrand is even. Finally by Proposition \ref{phaseint} we also have
%\begin{equation}
%\ph\zeta'(\medio+it)=\pi+\int_0^t\frac{\zeta''(\medio+ix)}{\zeta'(\medio+ix)}\,dx
%\end{equation}
%so that $\ph\zeta'(\frac12+it)$ is odd.
%\end{proof}

\begin{proposition}\label{anotherform}
Choosing  the phase of the real analytic function 
$t\mapsto Z'(t)-Z(t)\vartheta'(t)$ to be $=\pi/2$ at $t=0$ we will have
\begin{equation}\label{kappaph}
\kappa(t)=-\frac{1}{\pi}\ph(Z'(t)-iZ(t)\,\vartheta'(t)).
\end{equation}
\end{proposition}

\begin{proof} Choosing appropriately the phase of  $(Z'-iZ\vartheta')$, it follows from
\eqref{zetavartheta} that 
\begin{displaymath}
\vartheta(t)+\ph\zeta'(\medio+it)=\frac{\pi}{2}+\ph(Z'(t)-iZ(t)\vartheta'(t)).
\end{displaymath}
Thus, by \eqref{E:kappa1}
\begin{displaymath}
\kappa(t)=-\frac{1}{\pi}\ph(Z'(t)-iZ(t)\vartheta'(t)).
\end{displaymath}
(We have $Z'(0)-iZ(0)\vartheta'(0)=-i\,3.92264\dots$, so that we have 
to take the phase of $(Z'-iZ\vartheta')$ equal to $\pi/2$ at $t=0$.)
\end{proof}

\begin{proposition}\label{katheta}
We have $\kappa(a_\vartheta)=0$.
\end{proposition}

\begin{proof}
For $z$ not equal to  zero let, as usual, $\Arg
z$ be the determination of the argument of $z$ with $-\pi < \Arg z\le  \pi$.

By Proposition \ref{anotherform}
for every interval $I$ on which $Z(t)\vartheta'(t)\ne0$ there will
exist an integer $n_I$ such that
\begin{displaymath}
\kappa(t)=-\frac{1}{\pi}\Arg(Z'(t)-iZ(t)\vartheta'(t)) +n_I.
\end{displaymath}

In particular this applies to the interval $I=(0,a_\vartheta)$. To
determine $n_I$ in this case observe that $\kappa(0)=-\medio$,
$Z'(0)=0$, $Z(0)<0$ and $\vartheta'(0)<0$, and it follows that $n_I=-1$.

Then choose $\varepsilon>0$ small enough. At the point $t=a_\vartheta-\varepsilon$ we have $Z'(t)<0$,
$Z(t)<0$, and  $\vartheta'(a_\vartheta)=0$. Since $\kappa(t)$
is continuous and $\vartheta'(t)<0$ we get $\Arg(Z'(t)-iZ(t)\vartheta'(t))$ near 
$-\pi$ when $t=a_\vartheta-\varepsilon$. Taking limits for $\varepsilon\to0^+$ we get
$\kappa(a_\vartheta)=0$, as asserted.
\end{proof}

\begin{proposition}\label{P:17}
For each natural number $n$ we have  $\kappa(\gg_n)=n$.
\end{proposition}

\begin{proof}
Assuming that $\kappa(\gg_k)=k$  for $k\le n$  we will show 
that $\kappa(\gg_{n+1})=n+1$. The case $n=0$ is slightly different,
but similar. We assume now that $n\ge1$. 

In the interval $I_n=(\gg_n,\gg_{n+1})$ we have $Z(t)\vartheta'(t)\ne0$. 
Therefore 
\begin{displaymath}
\sgn(Z(t)\vartheta'(t))=\sgn(Z(t))=\nu.
\end{displaymath}
By Proposition \ref{anotherform} 
there is an integer $m$ such that 
\begin{displaymath}
\kappa(t)=m-\frac{1}{\pi}\Arg(Z'(t)-iZ(t)\vartheta'(t)),\qquad t\in I_n.
\end{displaymath}
For $y\ne 0$ we have
\begin{displaymath}
\Arg(x-iy)=\begin{cases}
-\arccos\frac{x}{\sqrt{x^2+y^2}} & y>0\\
\arccos\frac{x}{\sqrt{x^2+y^2}} & y<0.
\end{cases}
\end{displaymath}
Therefore,
\begin{equation}\label{equalitywithm}
\kappa(t)=m+\frac{\nu}{\pi}\arccos\frac{Z'(t)}{\sqrt{Z'(t)^2+Z(t)^2\vartheta'(t)^2}}\,\qquad t\in I_n.
\end{equation}
Then if $\mu=\sgn(Z'(t))$ we will have
\begin{displaymath}
\kappa(t)=m+\frac{\nu}{\pi}\arccos\frac{\mu}{\sqrt{1+\frac{Z(t)^2}{Z'(t)^2}\vartheta'(t)^2}}.
\end{displaymath}
For $t>\gg_n$ and $t\to\gg_n$ we have
for some $A>0$,  $C>0$ and an integer $\omega\ge1$
\begin{align*}
Z(t)&=\nu A (t-\gg_n)^\omega+\Orden(t-\gg_n)^{\omega+1}\\
Z'(t)&=\nu\omega A(t-\gg_n)^{\omega-1}+\Orden(t-\gg_n)^{\omega}\\
\frac{Z(t)}{Z'(t)}&=\frac{1}{\omega}(t-\gg_n)+\Orden(t-\gg_n)^2\\
\frac{Z(t)^2}{Z'(t)^2}\vartheta'(t)^2&=\frac{C^2}{\omega^2}(t-\gg_n)^2+\Orden(t-\gg_n)^3.
\end{align*}
Therefore, in a small interval to the right of $\gg_n$ the sign of $Z(t)$ is the same as 
the sign of $Z'(t)$, so that $\mu=\nu$.  Hence for $\gg_n<t<\gg_n+\delta$ we have
\begin{equation}\label{E:determination}
\kappa(t)=m+\frac{\nu}{\pi}\arccos\Bigl\{\nu\Bigl(1-\frac{1}{2}\frac{C^2}{\omega^2}(t-\gg_n)^2+\Orden(t-\gg_n)^3\Bigr)\Bigr\}.
\end{equation}
Observe that for small $x>0$  we have 
\begin{displaymath}
\arccos(1-x)=\sqrt{2}\sqrt{x}+\Orden(x^{3/2}),\qquad 
\arccos(-1+x)=\pi-\sqrt{2}\sqrt{x}+\Orden(x^{3/2}).
\end{displaymath}
It follows that for $\nu=1$
\begin{displaymath}
\kappa(t)=m+\frac{1}{\pi}\frac{C}{\omega}(t-\gg_n)+\Orden(t-\gg_n)^2
\end{displaymath}
and for $\nu=-1$
\begin{displaymath}
\kappa(t)=m-1+\frac{1}{\pi}\frac{C}{\omega}(t-\gg_n)+\Orden(t-\gg_n)^2.
\end{displaymath}
Taking limits for  $t\to \gg_n^+$ we get
\begin{displaymath}
n=\kappa(\gg_n)=\begin{cases}
m & \text{when $\nu=1$},\\
m-1 & \text{when $\nu=-1$}. 
\end{cases}
\end{displaymath}

Having determined $m$ we move $t$ to the other extreme of the interval $I_n$ in 
\eqref{equalitywithm}.  Therefore, now $\gg_{n+1}-\delta<t<\gg_{n+1}$ with 
$\delta$ small enough. We still have $\sgn Z(t)=\nu$, so that
$Z(t)=\nu B(t-\gg_{n+1})^\varpi$ with $B>0$.  As before we will get
$\frac{Z(t)}{Z'(t)}=\frac{1}{\varpi}(t-\gg_{n+1})$, but in this case this means
that $\sgn(Z'(t))=-\sgn(Z(t))$ so that $\mu=-\nu$, where now   $\mu$ is the sign
of $Z'(t)$ for $\gg_{n+1}-\delta<t<\gg_{n+1}$.
Hence in this case the analogue of \eqref{E:determination}
is 
\begin{equation}
\kappa(t)=m+\frac{\nu}{\pi}\arccos\Bigl\{-\nu\Bigl(1-\frac{1}{2}\frac{C'^2}{\varpi^2}(t-\gg_{n+1})^2+\Orden(t-\gg_{n+1})^3\Bigr)\Bigr\}.
\end{equation}
It follows that for $\nu=1$
\begin{displaymath}
\kappa(t)=m+1-\frac{1}{\pi}\frac{C'}{\varpi}(\gg_{n+1}-t)+\Orden(t-\gg_{n+1})^2.
\end{displaymath}
Taking limits for $t\to\gg_{n+1}$ we get
\begin{displaymath}
\kappa(\gg_{n+1})=m+1=n+1=\kappa(\gg_n)+1
\end{displaymath}
and for $\nu=-1$
\begin{displaymath}
\kappa(t)=m-\frac{1}{\pi}\frac{C'}{\varpi}(\gg_{n+1}-t)+\Orden(t-\gg_{n+1})^2
\end{displaymath}
so that in this case
\begin{displaymath}
\kappa(\gg_{n+1})=m=n+1=\kappa(\gg_n)+1.
\end{displaymath}
\end{proof}

\begin{corollary}\label{integervalues}
The function $\kappa(t)$ takes integer values only in the following cases:
$\kappa(a_\vartheta)=0$,  $\kappa(-a_\vartheta)=-1$,  $\kappa(\xi_n)=n$, $\kappa(-\xi_n)=-n-1$
for all natural numbers $n$. 
\end{corollary}

\begin{proof}
Since $\kappa(t)+\medio$ is an odd function we  get 
$\kappa(-t)=-\kappa(t)-1$, so that 
$\kappa(-a_\vartheta)=-1$ and $\kappa(-\gg_n)=-n-1$.

Assuming that $\kappa(t)\in\Z$, by \eqref{E:defk} we must have 
$\vartheta'(t)\zeta(\frac12+it)=0$ (recall that if $\zeta'(\frac12+it)=0$ then 
$\zeta(\medio+it)=0$ so that the quotient  $\zeta(\frac12+it)/\zeta'(\medio+it)$ is equal to $0$ in this case). 
By Corollary \ref{convexity}, for $t>0$, we have $\vartheta'(t)=0$ only 
for $t=a_\vartheta$. By definition the positive real numbers $t$ such that 
$\zeta(\frac12+it)=0$ are the numbers $\gg_n$. This proves that $\kappa(t)$ is an 
integer only at the points indicated. 
\end{proof}

\begin{corollary}
For $n=1$, $2$, \dots the number $\gg_n$ is the unique solution  of the 
equation $\kappa(t)=n$.
\end{corollary}

If we assume the RH and that the zeros are simple, we get that $\gamma_n$
is the only solution of the equation $\kappa(t)=n$. 

Define $\gg_0=a_\vartheta$, $\gg_{-1}=-a_\vartheta$, $\gg_{-n}=-\gg_{n-1}$, so that
for all integers $n\in\Z$ we have $\kappa(\gg_n)=n$.  With these notations we have

\begin{proposition}\label{monotonous}
For any integer $n\in\Z$ and $t$ with $\gg_n<t<\gg_{n+1}$ we have 
$\kappa(\gg_n)=n<\kappa(t)<n+1=\kappa(\gg_{n+1})$.
\end{proposition}

\begin{proof}
Since $t\ne \gg_m$ the value $\kappa(t)$ is not an integer. If $\kappa(t)<n$, since
$\kappa(x)$ is continuous, there will exist $t<t'<\xi_{n+1}$ with $\kappa(t')=n$, in contradiction
with Corollary \ref{integervalues}.
A similar reasoning rules out the possibility that $\kappa(t)>n+1$. 
\end{proof}

\begin{proposition}\label{P:kn0}
For $t>a_\vartheta$, let $N_{00}(t):=\card\{n\in\N: \gg_n\le  t\}$ 
be the number of real numbers $0<\xi\le t$ such that 
$\zeta(\medio+i\xi)=0$ counted without multiplicity. Then we have
\begin{equation}\label{Nkappa}
N_{00}(t)=\lfloor \kappa(t)\rfloor.\qquad   t>a_\vartheta.
\end{equation}
\end{proposition}

\begin{proof}
Since $t>a_\vartheta=\gg_0$ there is an integer $n\ge0$ such that 
$\gg_n\le t<\gg_{n+1}$. By definition $N_{00}(t)=n$ and by Proposition 
\ref{monotonous} $n\le \kappa(t)<n+1$ so that  $\lfloor \kappa(t)\rfloor =n$. 
\end{proof}

\begin{remark}
It is known \cite{BCY} that $N_0^*(T)$
the number of \emph{simple zeros on the line} to height $T$ satisfies
$\liminf_{T\to\infty}N_0^*(T)/N(T)\ge 0.4058$, where $N(T)$, as usual, denotes
the number of zeros $\beta+i\gamma$ of $\zeta(s)$ with $0<\gamma<T$ 
counted with their multiplicities. Since $\kappa(t)\ge N_0^*(t)$ we 
deduce that 
%\begin{equation}
$\liminf_{t\to\infty}\kappa(t)/N(t)\ge 0.4058$.
%\end{equation}
In \cite{BH},
assuming the RH (but not the simplicity of the zeros) this has been  improved to
\begin{equation}
\liminf_{t\to\infty}\kappa(t)/N(t)\ge 0.84665.
\end{equation}
\end{remark}

\begin{proposition}\label{midvalues}
For any real $t$ we have $\kappa(t)=\frac{2k+1}{2}$ with $k\in\Z$ if and 
only if $Z'(t)=0$ and $Z(t)\ne0$. 
\end{proposition}

%\begin{proof}
%According to \eqref{kappaph}, we have $\kappa(t)=\frac{2k+1}{2}$ if and only if 
%\begin{displaymath}
%\ph(Z'(t)-iZ(t)\vartheta'(t))=-(2k+1)\frac{\pi}{2}.
%\end{displaymath} 
%Therefore, if and only if
%$Z'(t)-iZ(t)\vartheta'(t)\in i\R$ and
%$Z'(t)-iZ(t)\vartheta'(t)=0$ only  when 
%$Z(t)=Z'(t)=0$, because 
%$Z'(a_\vartheta)=-Z'(-a_\vartheta)=-0.18838\cdots$ do not vanish. 
%But $Z(t)=Z'(t)=0$ can only happens
%when $t=\xi_n$ for some $n\in\Z$, and in this case $\kappa(t)\in\Z$. 
%
%Therefore,
%$\kappa(t)=\frac{2k+1}{2}$ if and only if $Z'(t)=0$ and $Z(t)\ne0$. 
%\end{proof}

\begin{proof}
The function $\vartheta'(t)$ only vanishes at $t=\pm a_\vartheta$ and at these
points the function $Z'(t)$ does not vanish ($Z'(a_\vartheta)=-Z'(-a_\vartheta)=-0.18838\cdots$). 
Hence $Z'(t)-iZ(t)\vartheta'(t)=0$ only at a point where $Z(t)=Z'(t)=0$. 
Since $Z(t)=0$ there exists $n$ with $t=\xi_n$. By Corollary \ref{integervalues} 
we know that at this point  $\kappa(t)\in\Z$.

Let $t$ be a point where $Z(t)\ne 0$ but $Z'(t)=0$, then 
$Z'(t)-iZ(t)\vartheta'(t)\in i\R^*$ and by \eqref{kappaph} we have
$\kappa(t)=-\frac{1}{\pi}\ph(Z'(t)-iZ(t)\vartheta'(t))=k+\frac{1}{2}$ for some
$k\in\Z$.

If, on the other hand, we assume $\kappa(t)=\frac{2k+1}{2}$, then again by
\eqref{kappaph}, $\ph(Z'(t)-iZ(t)\vartheta'(t))=-(2k+1)\frac{\pi}{2}$, so that
$Z'(t)-iZ(t)\vartheta'(t)\in i\R$, and certainly we will have $Z'(t)=0$ and
as we have seen $Z(t)\ne0$. 
\end{proof}

\section{Hypothesis \SP\  and its consequences.}

One may verify that
$\kappa'(0)$ is negative ($=-0.444 016\dots$). In fact
$\kappa'(t)$ is negative for all $t$ with 
\begin{displaymath}
|t|<a_\kappa
= 0.77985\,35753\,38836\,03051\,82092\,08122\,53710\,71856\,73276\,80740\,38626\,70020\dots
\end{displaymath}
We will prove in Proposition \ref{RHimpliesP} that, assuming the 
RH, $\kappa'(t)>0$ for $t>a_\kappa$.  But we are unable to prove the RH assuming 
$\kappa'(t)>0$ for $t>a_\kappa$.   However, this appears to be a realistic hypothesis 
(weaker than the RH):

\begin{hypothesis}
$\kappa'(t)\ge0$ for $t>a_\kappa$.
\end{hypothesis}

Some of our propositions will depend on this hypothesis.
We will attach the symbol  \SP\  to every proposition or theorem whose
proof depends on this  hypothesis.

\begin{proposition}[\SP]
For each integer $n\in\Z$, with $n\ge0$,  there is a unique real number $\eta_{n+2}$ such that 
$\gg_{n}<\eta_{n+2}<\gg_{n+1}$, and $Z'(\eta_{n+2})=0$.  The number $\eta_{n+2}$ is the unique 
solution to the equation $\kappa(t) = n+\frac12$.
\end{proposition}

\begin{figure}[H]
  \includegraphics[width=\hsize]{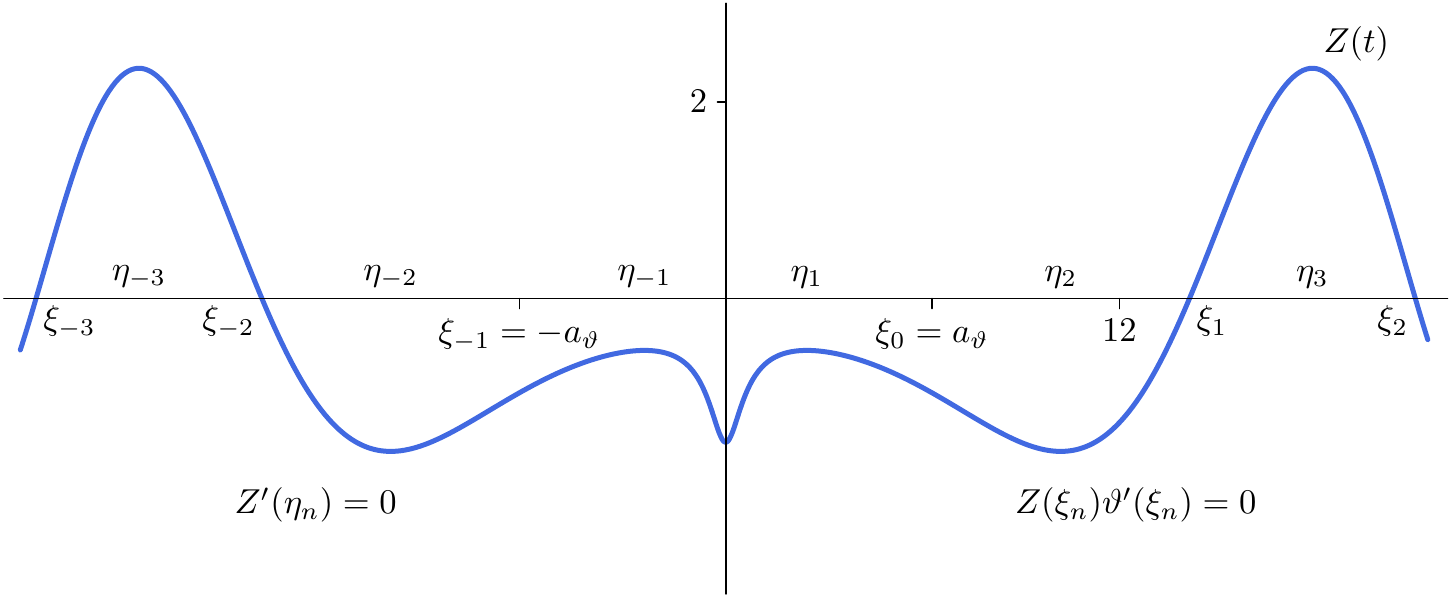}
  \caption{$Z(t)$ near the origin. }
  \label{F:F2}
\end{figure}

\begin{proof}
Since $a_\kappa<a_\vartheta$ and $\kappa(t)$ is real analytic,  the hypothesis \SP\ implies 
that $\kappa(t)$, being  analytic is strictly increasing for $t> \gg_0:=a_\vartheta$.
Therefore,  for $n\ge0$ the function $\kappa(t)$ is strictly increasing in the
interval $(\gg_{n},\gg_{n+1})$, so that
there is only one solution to the equation $\kappa(t)=n+\frac12$.

By Proposition \ref{midvalues} the solution $t=\eta_{n+2}$ to the above equation  is the only 
possible solution of the equation $Z'(t)=0$ in this interval. 
\end{proof}

For $n=-1$ we may check numerically  that 
$t=0$ and $t= \pm\,2.4757266\dots$ are solutions to $Z'(t)=0$ in the interval 
$(\xi_{-1}, \xi_0)=(-a_\vartheta,a_\vartheta)$. 

Using the above it is easy to see that the points where $\kappa(t) = n+\frac12$
are the following:
\smallskip

\noindent (a) Three points in the interval $(\gg_{-1}, \gg_0)=
(a_{-\vartheta}, a_\vartheta)$. These are
$\eta_1=-\eta_{-1}=2.47572\dots$, and $\eta_0=0$ at which 
$\kappa(\eta_{-1})=\kappa(\eta_0)=\kappa(\eta_1)=-\frac12$.
\smallskip

\noindent (b) A point $\eta_2\in(\gg_0,\gg_1)$, with $\eta_2=10.21207\dots$ 
at which $\kappa(\eta_2)=\frac12$
and its symmetric $\eta_{-2}=-\eta_2$ at which $\kappa(-\eta_2)=-\frac32$. 
\smallskip

\noindent (c) For each integer $n\ge1$ a unique point 
$\eta_{n+2}\in(\gg_n,\gg_{n+1})$ at which 
$\kappa(\eta_{n+2})=n+\frac{1}{2}$. Its symmetrical 
$\eta_{-n-2}\in(\gg_{-n-2}, \gg_{-n-1})$ with
$\eta_{-n-2}=-\eta_{n+2}$ and $\kappa(\eta_{-n-2})=-\frac{2n+3}{2}$.

\medskip

One may verify that the minimal 
value $a_\gamma$ of $\kappa(t)$ is
\begin{displaymath}
a_\gamma:=\kappa(a_\kappa)=
-0.67025\,97987\,68599\,50288\,39164\,11968\,66744\,74803\,92790\,09743\,49173\,\dots
\end{displaymath}
Since $\kappa$ is strictly increasing on $(a_\kappa,+\infty)$ with values in
$(a_\gamma,\infty)$ we may define $\gamma(u)$ for $u>a_\gamma$ as the inverse function 
of $\kappa(t)$.   $\gamma(u)$ is a real analytic function on $(a_\gamma,\infty)$ and 
we will have
\begin{displaymath}
\gamma(n) =\xi_n,\qquad \gamma(n+\medio)=\eta_{n+2}, \qquad n\ge0, \quad\text{assuming \SP}.
\end{displaymath}
Of course, assuming the RH with simple zeros we will have $\gamma(n)=\gamma_n$. 

\begin{proposition}
For $t\in\R$ not a multiple zero of $Z(t)$ we have
\begin{equation}\label{eqkappaprime}
\kappa'=\frac{1}{\pi}\frac{ZZ'\vartheta''+(Z')^2\vartheta'-ZZ''\vartheta'}{(Z')^2+(Z\vartheta')^2},
\end{equation}
where for short we have written $\kappa'$ for $\kappa'(t)$,   $Z$ for $Z(t)$,  etc. 
Therefore 
\begin{equation}
\SP\quad\Longleftrightarrow \quad ZZ'\vartheta''+(Z')^2\vartheta'-ZZ''\vartheta'\ge0 \quad 
\text{for $t>a_\kappa$}.
\end{equation}
\end{proposition}

\begin{proof}
For $\xi_n<t<\xi_{n+1}$ we have \eqref{equalitywithm} for some constant $m$. Differentiating
and simplifying we get \eqref{eqkappaprime}.
Since $\kappa'$ is real analytic the equality is true because we are not dividing by $0$. 
\end{proof}

\begin{figure}[H]
  \includegraphics[width=\hsize]{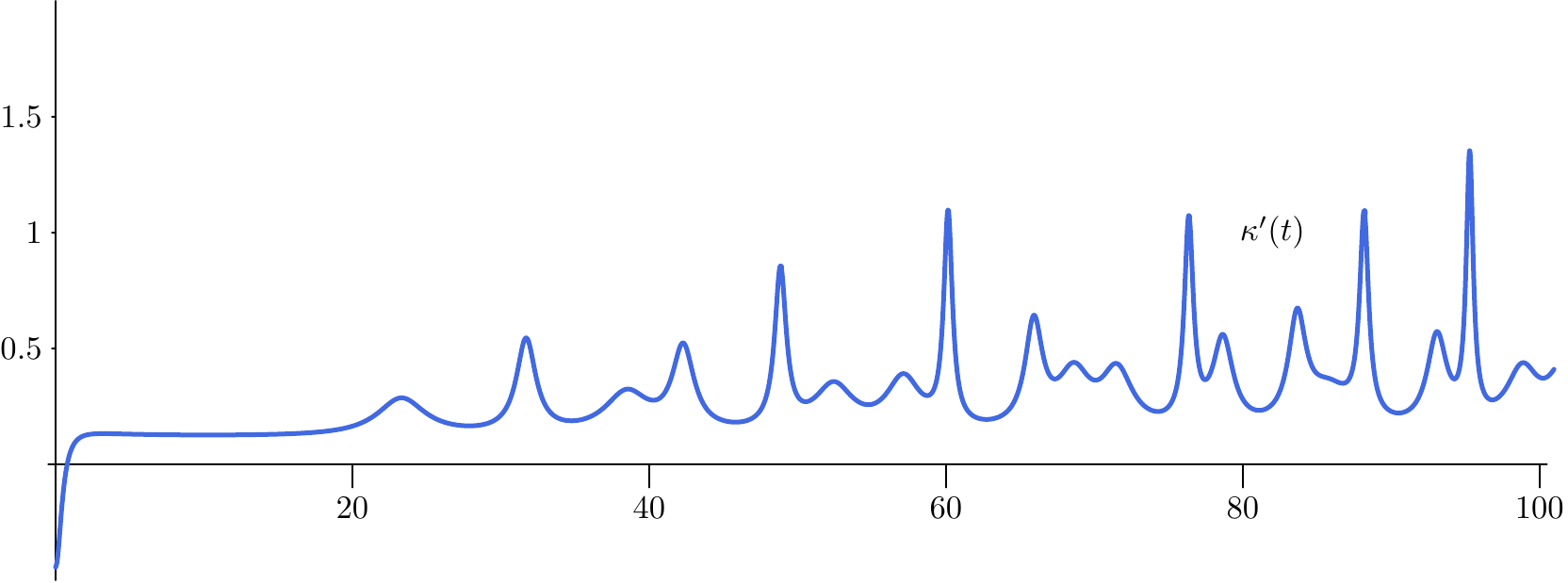}
  \caption{$\kappa'(t)$}
  \label{F:F3}
\end{figure}

\section{Connection of $\kappa'(t)$ with the zeros of $\zeta'(s)$.}\label{S:conection}

We will need some known facts (see \cite{S}, \cite{LM}*{Theorem 9}, \cite{B} and 
\cite{T}*{Theorem 11.5(C)}) about
the zeros of $\zeta'(s)$.

\begin{proposition}

\begin{itemize}
\item[\emph{(a\hspace{0.05cm})}] For $n\ge 1$ there is a unique real solution $a_n$ of
$\zeta'(s)=0$ with $-2n-2< a_n<-2n$, and there are no other  zeros of
$\zeta'(s)$ in $\sigma\le 0$.

\item[\emph{(b)}] Let $\rho'=\beta'+i\gamma'$ denote the non real zeros of
$\zeta'(s)$, and let $N_1(T)$ denote the number of non real zeros of
$\zeta'(s)$ with $0<\gamma<T$.  Then
\begin{displaymath}
N_1(T)=\frac{T}{2\pi}\log \frac{T}{4\pi}-\frac{T}{2\pi}+\Orden(\log
T).
\end{displaymath}

\item[\emph{(c)}] We have $0< \beta'\le E$ where $E\le 3$ is a constant.  The
Riemann Hypothesis is equivalent to $\zeta'(s)$ having no zeros in
$0<\sigma<1/2$.
\end{itemize}
\end{proposition}

We will use $\rho'=\beta'+i\gamma'$ to denote a typical complex
zero of $\zeta'(s)$. Other times we prefer to denote  by
$\rho'_n=\beta'_n+i\gamma'_n $ the sequence of  zeros with
$\gamma'_n>0$ numbered in such a way that
$0<\gamma'_1\le\gamma'_2\le \cdots$, with the understanding that 
the ordinate of a zero of multiplicity $m$ appears $m$ times consecutively
in this sequence.

\begin{proposition} We have the following Mittag-Leffler expansion
\begin{equation}\label{E:MittagLeffler}
\frac{\zeta''(s)}{\zeta'(s)}=a-\frac{2}{s-1}+\sum_{n=1}^\infty
\Bigl(\frac{1}{s-a_n}+\frac{1}{a_n}\Bigr)+ \sum_{\rho'}
\Bigl(\frac{1}{s-\rho'}+\frac{1}{\rho'}\Bigr)
\end{equation}
where the $a_n$ are the real, and $\rho'_n$ the complex  zeros of
$\zeta'(s)$, and $a=0.18334\dots$ is a constant ($=-2+\zeta''(0)/\zeta'(0)$).
\end{proposition}

\begin{proof} The entire function $f(s)=(s-1)^2\zeta'(s)$ has  the
same order as $(s-1)\zeta(s)$ so that $f(s)$ is an entire function
of order $1$.

From the above results about the zeros of $\zeta'(s)$ it follows
easily that the exponent of convergence of the zeros of $f(s)$ is
$1$. Also, the series $\sum_{n=1}^\infty 1/|a_n|$ is divergent. Thus
we have
\begin{equation}\label{E:product}
\zeta'(s)=e^{as+b}(s-1)^{-2} \prod_{n=1}^\infty
\Bigl(1-\frac{s}{a_n}\Bigr)e^{s/a_n} \prod_{\rho'}
\Bigl(1-\frac{s}{\rho'}\Bigr)e^{s/\rho'}
\end{equation}
for some constants $a$ and $b$.

Now we take logarithms and differentiate to get
\eqref{E:MittagLeffler}. At the point $s=0$ we obtain the equality
\begin{displaymath}
\frac{\zeta''(0)}{\zeta'(0)}=a+2
\end{displaymath}
from which we get the numerical value for $a$ given in the statement.
\end{proof}

\begin{remark}
It can be proved that
\begin{displaymath}
\frac{\zeta''(0)}{\zeta'(0)}=
\frac{\pi^2}{12\log2\pi}-\frac{\gamma^2+2\gamma_1}{\log2\pi}+\log2\pi
\end{displaymath}
where $\gamma$ (the Euler constant) and $\gamma_1$ are  Stieltjes
constants appearing as coefficients in the Mittag--Leffler
expansion of $\zeta(s)$ at the point $s=1$.
\end{remark}

\begin{remark}
The constant $b$ in equation \eqref{E:product} is determined by
$e^b=\zeta'(0)=-\medio\log(2\pi)$. So $b$ is complex. 
\end{remark}

\begin{proposition}\label{P: kappaprime}
We have
\begin{equation}\label{E:kappaprime}
\pi\kappa'(t)=A+f(t)+\sum_{\rho'=\beta'+i\gamma'}
\frac{\beta'-1/2}{(1/2-\beta')^2+(t-\gamma')^2}
\end{equation}
where $A$ is a constant and $f(t)$ is a bounded continuous function
such that $f(t)=\Orden(t^{-1})$ as $t\to+\infty$.
\end{proposition}

\begin{remark}
The exact definition of $f(t)$ is given in \eqref{def f}.
\end{remark}

\begin{remark}
In Corollary \ref{main} we will prove that $A=\frac12\log2$.
\end{remark}

\begin{proof}
From \eqref{E:kappa} we get
\begin{equation}\label{E:kappaprime2}
\pi\kappa'(t)=-\vartheta'(t)-\Re\frac{\zeta''(\medio+it)}{\zeta'(\medio+it)}.
\end{equation}

Now in \eqref{E:MittagLeffler} we put $s=1/2+it$ and take real parts
\begin{multline*}
\Re\frac{\zeta''(1/2+it)}{\zeta'(1/2+it)}= a+\frac{4}{1+4t^2}+
\sum_{n=1}^\infty
\Bigl(\frac{1/2-a_n}{(1/2-a_n)^2+t^2}+\frac{1}{a_n}\Bigr)\\+
\sum_{\rho'=\beta'+i\gamma'}
\Bigl(\frac{1/2-\beta'}{(1/2-\beta')^2+(t-\gamma')^2}+
\frac{\beta'}{\beta'{}^2+\gamma'{}^2}\Bigr).
\end{multline*}
Hence from \eqref{E:kappaprime2} and \eqref{E:thetaprime} we get
\begin{multline*}
\pi\kappa'(t)=A-\frac{2}{1+4t^2}+\sum_{n=1}^\infty
\Bigl(\frac{2(4n+1)}{(4n+1)^2+4t^2}-\frac{1/2-a_n}{(1/2-a_n)^2+t^2}\Bigr)\\
+\sum_{\rho'=\beta'+i\gamma'}
\frac{\beta'-1/2}{(1/2-\beta')^2+(t-\gamma')^2}
\end{multline*}
where
\begin{equation}\label{E:A}
A=\frac{1}{2}(\gamma+\log\pi)-a-
\sum_{n=1}^\infty\frac{2\beta'_n}{\beta'_n{}^2+\gamma'_n{}^2}-
\sum_{n=1}^\infty\Bigl(\frac{1}{2n}+\frac{1}{a_n}\Bigr).
\end{equation}
We put
\begin{equation}\label{def f}
f(t)=-\frac{2}{1+4t^2}+\sum_{n=1}^\infty
\Bigl(\frac{2(4n+1)}{(4n+1)^2+4t^2}-\frac{1/2-a_n}{(1/2-a_n)^2+t^2}\Bigr).
\end{equation}
Now observe that the terms of the sum can be written as
\begin{displaymath}
\frac{(2n+1/2)}{(2n+1/2)^2+t^2}-\frac{1/2-a_n}{(1/2-a_n)^2+t^2}=
\int_{2n+1/2}^{1/2-a_n}
\frac{x^2-t^2}{(x^2+t^2)^2}\,dx.
\end{displaymath}
The intervals $(2n+1/2, 1/2-a_n)$ do not intersect, so that for $|t|<T$
the absolute values of the terms of the sum are bounded by
\begin{multline*}
\sum_{n=1}^\infty
\Bigl|\frac{2(4n+1)}{(4n+1)^2+4t^2}-\frac{1/2-a_n}{(1/2-a_n)^2+t^2}\Bigr|\le
\int_{5/2}^{+\infty}\frac{|x^2-t^2|}{(x^2+t^2)^2}\,dx\\
\le
\int_{5/2}^{+\infty}\frac{T^2+x^2}{x^4}\,dx<+\infty.
\end{multline*}
This proves  that $f(t)$ is a continuous function.

Also for $t>1$ we will have
\begin{equation}\label{cota}
|f(t)|\le\frac{2}{1+4t^2}+\int_0^{+\infty}\frac{|x^2-t^2|}{(x^2+t^2)^2}\,dx=
\frac{2}{1+4t^2}+\frac{1}{t}.
\end{equation}
\end{proof}

\begin{remark}
It can be shown that the zero $a_{n-1}$ contained in $(-2n, -2n+2)$ satisfies
$2n+a_{n-1}\sim 1/\log(n/\pi)$. This information can be used to show that
$f(t)=\Orden(1/t\log t)$. In this way we may improve the error term in 
\eqref{pikappa} from $\Orden(\log t)$ to $\Orden(\log\log t)$. 
\end{remark}

We introduce some  notation: if $t>0$ and $n\ge1$ let $\angles(t,\rho'_n)=\angles_n(t)$
be the angle at $\rho'_n$ of the triangle with vertices at
$\rho'_n=\beta'_n+i\gamma'_n$, $1/2-it$ and $1/2+it$. We consider
this angle expressed in radians  positive if $\beta'_n>1/2$ and
negative if $\beta'_n<1/2$, and we put $\angles_n(t)=0$ when
$\beta'_n=\medio$.  In other words with $s=\medio+it$ and $\rho'=\beta'+i\gamma'$
we have
\begin{equation}
\angles(t, \beta'+i\gamma')=\arctan\frac{t-\gamma'}{\beta'-\frac12}+
\arctan\frac{t+\gamma'}{\beta'-\frac12}=\Arg\frac{\overline{s}-\rho'}{s-\rho'}\qquad  (\beta'\ne\medio).
\end{equation}

\begin{proposition} For $t>0$ we have
\begin{equation}\label{pikappa}
\pi\kappa(t)=A t+\sum_{n=1}^\infty\angles_n(t)+\Orden(\log t)
\end{equation}
where the sum is extended over all zeros
$\rho'_n=\beta'_n+i\gamma'_n$ of $\zeta'(s)$ with $\gamma'_n>0$.
\end{proposition}

\begin{proof} By \eqref{E:kappa} and \eqref{E:kappaprime} we have
\begin{displaymath}
\pi\kappa(t)=-\frac{\pi}{2}+At+ \int_0^t f(x)\,dx+\int_0^t
\sum_{\rho'=\beta'+i\gamma'}
\frac{\beta'-\medio}{(\medio-\beta')^2+(x-\gamma')^2}\,dx.
\end{displaymath}

Observe that if $\beta'=\medio$ then the corresponding term does not
contribute to   the sum.

Thus
\begin{displaymath}
\pi\kappa(t)=-\frac{\pi}{2}+ At+ \int_0^t f(x)\,dx+
\sum_{\rho'=\beta'+i\gamma'}
\Bigl\{\arctan\frac{t-\gamma'}{\beta'-\medio}+
\arctan\frac{\gamma'}{\beta'-\medio}\Bigr\}
\end{displaymath}
where the terms with $\beta'_n=\medio$ should be omitted.
It is easy
to see that the sum of the terms corresponding to
$\rho'_n=\beta'_n+i\gamma'_n$ and
$\overline{\rho'_n}=\beta'_n-i\gamma'_n$ add up to exactly $\angles_n(t)$.
(This is the reason we made the convention about the sign of
$\angles_n(t)$). Thus we arrive at
\begin{displaymath}
\pi\kappa(t)=-\frac{\pi}{2}+ At+ \int_0^t f(x)\,dx+ \sum_{n=1}^\infty
\angles_n(t).
\end{displaymath}
Now, since
$f(t)=\Orden(t^{-1}) $ we can write this as
\begin{displaymath}
\pi\kappa(t)= At+ \sum_{n=1}^\infty\angles_n(t)+\Orden(\log t).
\end{displaymath}
\end{proof}

\section{Counting the zeros of $\zeta(s)$.}\label{counting}

The exact value of the constant $A$ in \eqref{pikappa} can be obtained in 
two ways. One by computing the constants in the Mittag-Leffler expansion 
of related functions and the second, more interesting for us,
by comparing two different counts of the
number of zeros of $\zeta(s)$. We will present this second proof. 
We need some definitions:

\def\O{\phantom{0}}

Let
\begin{align*}
N_-(T)&=\#\{\rho=\beta+i\gamma \O|\O \zeta(\rho)=0,\O
\beta<\medio,\O
0<\gamma\le T\}\\
N_0(T)&=\#\{\rho=\medio+i\gamma \O|\O\zeta(\rho)=0,\O
0<\gamma\le T\}\\
N'(T)&=\#\{\rho'=\beta'+i\gamma' \O|\O\zeta'(\rho')=0,\O 0<\gamma'\le T\}\\
N'_-(T)&=\#\{\rho'=\beta'+i\gamma' \O|\O\zeta'(\rho')=0,\O
\beta'<\medio,\O 0<\gamma'\le T\}\\
N'_0(T)&=\#\{\rho'=\medio+i\gamma' \O|\O\zeta'(\rho')=0,\O
0<\gamma'\le T\}\\
N'_+(T)&=\#\{\rho'=\beta'+i\gamma'\O|\O\zeta'(\rho')=0,\O
\beta'>\medio,\O 0<\gamma'\le T\}.
\end{align*}
In all these cases, as usual, we count the zeros with their
multiplicities. But we also need to consider another count
$N_{00}(T)$ which is the number of real numbers $0<\gg\le T$ such
that $\zeta(\medio+i\gg)=0$, but in these cases we do not count
multiplicities.

Taking account of Proposition \ref{P:doblezeros} it is clear that
\begin{equation}\label{zeroszetap}
N_0(T)-N_{00}(T)=N'_0(T)
\end{equation}
which equals the number of zeros of
$\zeta'(s)$ on the critical line with  $0<\gamma'\le T$.

We know some relations between these counts:

\noindent (i) Backlund refining previous work of von Mangoldt  (see Edwards, 
\cite{E}*{Section 6.7}) gave a complete proof of  Riemann's assertion
\begin{equation}\label{Mangoldt}
N(T)=N_0(T)+2N_-(T)=
\frac{T}{2\pi}\log\frac{T}{2\pi}-\frac{T}{2\pi}+\Orden(\log T).
\end{equation}

\noindent (ii) Berndt \cite{B} proved the corresponding result for $\zeta'(s)$
\begin{equation}\label{Berndt}
N'(T)=N'_-(T)+N'_0(T)+N'_+(T)=\frac{T}{2\pi}\log\frac{T}{4\pi}-\frac{T}{2\pi}+\Orden(\log
T).
\end{equation}

\noindent (iii) Levinson and Montgomery \cite{LM} showed that
\begin{equation}\label{Levinson}
N_-(T)=N'_-(T)+\Orden(\log T).
\end{equation}

\noindent (iv) From our Proposition \ref{P:kn0} we get that
\begin{equation}\label{N00}
\kappa(t) =N_{00}(t)+\Orden(1).
\end{equation}

Finally, in Proposition \ref{kappaN}   we will prove a new
relation \eqref{E:kappaN}. First we prove two lemmas about the zeros
of $\zeta'(s)$.

\begin{table}[H]
\caption{First non-trivial zeros of $\zeta(s)$ and $\zeta'(s)$}
\begin{center}
\begin{tabular}{|c|c|}
\hline
$\beta_n+i\gamma_n$ & $\beta'_n+i\gamma'_n$\\
\hline
\noalign{\smallskip}
\hline
$0.5+i\,14.13472\,51417$ & $2.46316\,18694+ i\,23.29832\,04927$\\
$0.5+i\,21.02203\,96387$ & $1.28649\,68222+i\,31.70825\,00831$\\
$0.5+i\,25.01085\,75801$ & $2.30757\,00637+i\,38.48998\,31730$\\
$0.5+i\,30.42487\,61258$ & $1.38276\,36057+i\,42.29096\,45545$\\
$0.5+i\,32.93506\,15877$ & $0.96468\,56227+i\,48.84715\,99050$\\
$0.5+i\,37.58617\,81588$ & $2.10169\,99009+i\,52.43216\,12451$\\
$0.5+i\,40.91871\,90121$ & $1.89595\,97624+i\,57.13475\,31990$\\
%$0.5+i\,43.32707\,32809$ & $0.84873\,53281+i\,60.14084\,57820$\\
%$0.5+i\,48.00515\,08811$ & $1.20729\,56246+i\,65.91993\,28242$\\
%$0.5+i\,49.77383\,24776$ & $1.83294\,79316+i\,68.61107\,88271$\\
\hline
\end{tabular}
\end{center}
\label{default}
\end{table}%

\begin{lemma}
For $t>0$ we have
\begin{displaymath}
\sum_{0<\gamma'_n<t-1}\frac{1}{t-\gamma'_n}=\Orden(\log^2 t).
\end{displaymath}
\end{lemma}

\begin{proof} Put
\begin{displaymath}
N'(t)=\frac{t}{2\pi}\log\frac{t}{4\pi}-\frac{t}{2\pi}+R(t)
\end{displaymath}
where  $R(t)=\Orden(\log t)$ by Berndt's Theorem. The first zero of
$\zeta'(s)$ is $\rho'\approx 2.46316 +23.29832 i$ so that $N'(4\pi)=0$ and consequently $R(4\pi)=2$. We have
\begin{displaymath}
\sum_{0<\gamma'_n<t-1}\frac{1}{t-\gamma'_n}=\int_{4\pi}^{t-1}\frac{d
N'(x)}{t-x}=\frac{1}{2\pi}\int_{4\pi}^{t-1}\frac{\log(x/4\pi)}{t-x}\,dx+
\int_{4\pi}^{t-1}\frac{d R(x)}{t-x}
\end{displaymath}
\begin{displaymath}
\le\log(t/4\pi)\frac{\log(t-4\pi)}{2\pi}+R(t-1)+
\int_{4\pi}^{t-1}\frac{|R(x)|}{(t-x)^2}\,dx.
\end{displaymath}
Since $R(x)=\Orden(\log x)$  all the above terms are $\Orden(\log^2
t)$.
\end{proof}

\begin{lemma} For $t\to+\infty$ we have
\begin{displaymath}
\sum_{\gamma'_n>t+1}
\frac{1}{\gamma'_n{}^2-t^2}=\Orden\Bigl(\frac{\log^2 t}{t}\Bigr).
\end{displaymath}
\end{lemma}

\begin{proof} Retaining the notations of the previous lemma we have
\begin{displaymath}
\sum_{\gamma'_n>t+1}
\frac{1}{\gamma'_n{}^2-t^2}=\int_{t+1}^{+\infty}\frac{dN'(x)}{x^2-t^2}=
\frac{1}{2\pi}\int_{t+1}^{\infty}\frac{\log(x/4\pi)}{x^2-t^2}\,dx+
\int_{t+1}^{+\infty}\frac{dR(x)}{x^2-t^2}.
\end{displaymath}
For $t>4\pi$ the first integral is less than or equal to $C\log^2 t /t$:
\begin{displaymath}
\frac{1}{2\pi}\int_{t+1}^{\infty}\frac{\log(x/4\pi)}{x^2-t^2}\,dx\le\frac{2\log
t}{2\pi}\int_{t+1}^{t^2}\frac{dx}{x^2-t^2}+
\int_{t^2}^{+\infty}\frac{\sqrt{x}}{x^2-t^2}\,dx
\end{displaymath}
\begin{displaymath}
= \frac{\log t}{2\pi t}\Bigl(\log\frac{t-1}{t+1}+\log(2t+1)\Bigr)+
t^{-1/2}\int_t^{+\infty}\frac{\sqrt{y}}{y^2-1}\,dy \le \frac{(\log
t)^2}{\pi t}+\frac{4}{t}.
\end{displaymath}

Now we bound the second integral
\begin{displaymath}
\int_{t+1}^{+\infty}\frac{dR(x)}{x^2-t^2}=-\frac{R(t+1)}{2t+1}+\int_{t+1}^{+\infty}
\frac{R(x)}{(x^2-t^2)^2}2x\,dx.
\end{displaymath}
For $x>t+1$ we
have $x/(x^2-t^2)<1$, and $|R(x)|\le C\log x$. Thus we have
\begin{displaymath}
\int_{t+1}^{+\infty}\frac{dR(x)}{x^2-t^2}\le c_1\frac{\log
t}{t}+c_2\int_{t+1}^{+\infty}\frac{\log
x}{x^2-t^2}\,dx.
\end{displaymath}
Finally, this integral is bounded
exactly as the first integral.
\end{proof}

\begin{proposition}\label{kappaN}
For $t\to+\infty$
\begin{equation}\label{E:kappaN}
\pi\kappa(t)=At+\pi N'_+(t)-\pi N'_-(t)+\Orden(\log^2 t).
\end{equation}
\end{proposition}

\begin{proof}
By \eqref{pikappa} we have to show that
\begin{displaymath}
\sum_{\gamma'_n>0}\angles_n(t)=\pi N'_+(t)-\pi N'_-(t)+\Orden(\log^2
t).
\end{displaymath}
To this end we will show that
\begin{displaymath}
\sum_{\substack{\gamma'_n>0\\\beta'_n>\medio}}\angles_n(t)=\pi
N'_+(t)+\Orden(\log^2 t);\qquad
\sum_{\substack{\gamma'_n>0\\\beta'_n<\medio}}\angles_n(t)=-\pi
N'_-(t)+\Orden(\log^2 t).
\end{displaymath}
To simplify the notation we will write $\sideset{}{^+}\sum$ to
denote a sum restricted to $\beta'_n>\medio$ and
$\sideset{}{^-}\sum$ for a sum restricted to $\beta'_n<\medio$.

We split the sums into three terms
\begin{displaymath}
\sideset{}{^+}\sum_{\gamma'_n>0}\angles_n(t)=
\sideset{}{^+}\sum_{0<\gamma'_n<t-1}\angles_n(t)+
\sideset{}{^+}\sum_{|\gamma'_n-t|\le1}
\angles_n(t)+\sideset{}{^+}\sum_{\gamma'_n>t+1}\angles_n(t).
\end{displaymath}
The middle sum is $\Orden(\log t)$ because each term is (in absolute
value) less than $\pi$ and the number of terms is $\Orden(\log
t)$. In the first sum the summands are approximately $\pi$ (or
$-\pi$).  Thus we arrive at
\begin{displaymath}
\sideset{}{^+}\sum_{\gamma'_n>0}\angles_n(t)=\pi N'_+(t)
+\sideset{}{^+}\sum_{0<\gamma'_n<t-1}\bigl\{\angles_n(t)-\pi\bigr\} +
\sideset{}{^+}\sum_{\gamma'_n>t+1}\angles_n(t)+\Orden(\log t)
\end{displaymath}
and 
\begin{displaymath}
\sideset{}{^-}\sum_{\gamma'_n>0}\angles_n(t)=-\pi N'_-(t)
+\sideset{}{^-}\sum_{0<\gamma'_n<t-1}\bigl\{\angles_n(t)+\pi\bigr\} +
\sideset{}{^-}\sum_{\gamma'_n>t+1}\angles_n(t)+\Orden(\log t).
\end{displaymath}
It follows that
\begin{displaymath}
\sum_{\gamma'_n>0}\angles_n(t)=\pi N'_+(t)-\pi N'_-(t)+
\sum_{0<\gamma'_n<t-1}\bigl\{\angles_n(t)\pm\pi\bigr\}+
\sum_{\gamma'_n>t+1}\angles_n(t)+\Orden(\log t)
\end{displaymath}
where we must use the $+$ sign when $\beta'_n<\medio$ and the $-$
sign when $\beta'_n>\medio$.

Now for $0<\gamma'_n<t-1$ and $\beta'_n>\medio$ we have
\begin{displaymath}
0<\pi-\angles_n(t)<
\arctan\frac{\beta'_n-\medio}{t-\gamma'_n}+\arctan\frac{\beta'_n-\medio}{t+\gamma'_n}
<2\arctan\frac{\beta'_n-\medio}{t-\gamma'_n}<\frac{6}{t-\gamma'_n}
\end{displaymath}
and in the case $\beta'_n<\medio$ analogously
\begin{displaymath}
0<\pi+\angles_n(t)<
\arctan\frac{\medio-\beta'_n}{t-\gamma'_n}+\arctan\frac{\medio-\beta'_n}{t+\gamma'_n}
<2\arctan\frac{\medio-\beta'_n}{t-\gamma'_n}<\frac{1}{t-\gamma'_n}.
\end{displaymath}

Also, for $\gamma'_n>t+1$ and $\beta'_n>\medio$
\begin{displaymath}
0<\angles_n(t)=\arctan\frac{2(\beta'_n-\medio)t}{(\beta'_n-\medio)^2+\gamma'_n{}^2-t^2}
<\frac{6t}{\gamma'_n{}^2-t^2}
\end{displaymath}
and for $\beta'_n<\medio$ the absolute value $|\angles_n(t)|$ is
bounded by the same quantity.

Thus applying the two above lemmas we find that
\begin{displaymath}
\sum_{\gamma'_n>0}\angles_n(t)=\pi N'_+(t)-\pi N'_-(t)+ \Orden(\log^2
t).
\end{displaymath}
\end{proof}

\begin{corollary} For $t\to+\infty$
\begin{equation}\label{JuanJan}
N_{00}(t)=\frac{A}{\pi}t+N'_+(t)-N'_-(t)+\Orden(\log^2 t).
\end{equation}
\end{corollary}

\begin{proof}
Combine \eqref{E:kappaN} with \eqref{N00}.
\end{proof}

\begin{corollary}\label{main}
The constant $A$ is equal to $\frac12\log2$.
\end{corollary}

\begin{proof} Write $f(t)\circeq g(t)$ to denote that
$f(t)-g(t)=\Orden(\log^2 t)$. (In the same way as  congruences we
can operate with $\circeq$ as if it were an equality sign between
equivalence classes). With this notation we have
\begin{alignat*}{2}
N(t)&\circeq\frac{t}{2\pi}\log\frac{t}{2\pi}-\frac{t}{2\pi}&\quad&\text{by
\eqref{Mangoldt}}\\
&\circeq N_0(t)+2N_-(t)&\quad&\text{trivially}\\
&\circeq N_0(t)+2N'_-(t)&\quad&\text{by \eqref{Levinson}}\\
&\circeq N_0(t)-N_{00}(t)+\frac{A}{\pi}t+N'_+(t)+N'_-(t)&\quad&\text{by \eqref{JuanJan}}\\
&\circeq N_0(t)-N_{00}(t)+
\frac{A}{\pi}t-N'_0(t)+\frac{t}{2\pi}\log\frac{t}{4\pi}-\frac{t}{2\pi}&\quad&\text{by
\eqref{Berndt}}\\
&\circeq\frac{A}{\pi}t+\frac{t}{2\pi}\log\frac{t}{4\pi}-\frac{t}{2\pi}&\quad&\text{by
\eqref{zeroszetap}}
\end{alignat*}

We thus have
\begin{displaymath}
\frac{t}{2\pi}\log\frac{t}{2\pi}-\frac{t}{2\pi}\circeq 
\frac{A}{\pi}t+\frac{t}{2\pi}\log\frac{t}{4\pi}-\frac{t}{2\pi}
\end{displaymath}
from which we deduce
\begin{displaymath}
\frac{A}{\pi}t\circeq  \frac{t}{2\pi}\log2.
\end{displaymath}
Hence $A=\medio\log2$. 
\end{proof}

\begin{proposition}\label{RHimpliesP}
The Riemann hypothesis implies the Hypothesis $\SP$.
\end{proposition}

\begin{proof}
The Riemann hypothesis is equivalent to $\beta'>\medio$ for every zero
$\rho'=\beta'+i\gamma'$, and  it follows  by \eqref{E:kappaprime} that if
the Riemann hypothesis is true, then $\pi\kappa'(t)>A+f(t)$.  
Since $A=\frac12\log2$, applying \eqref{cota} 
we  easily see that $\kappa'(t)>0$ for $t>3.4$ if we assume the RH. 
It is clear that there is an
$a_\kappa\ge0$ such that $\kappa'(t)>0$ for $t>a_\kappa$ and $a_\kappa<3.4$.
\end{proof}

\section{Connections between the zeros of $\zeta(s)$ and $\zeta'(s)$.}\label{kappaprimeandzeros}

\begin{proposition}\label{kappaprimeatzeros}
Let $\frac12+i\xi$ be a zero of $\zeta(s)$ of multiplicity $\omega$ on the critical line, 
then 
\begin{equation}
\kappa'(\xi) =\frac{1}{\pi\omega}\vartheta'(\xi).
\end{equation}
\end{proposition}

\begin{proof}
Since $\zeta(\frac12+it)=e^{-i\vartheta(t)}Z(t)$ the function $Z(t)$ has a zero of 
multiplicity $\omega$ at $t=\xi$. 
Hence
$\lim_{t\to\xi}\frac{Z(t)}{Z'(t)} = 0$, and for $\omega\ge2$, 
$\lim_{t\to\xi}\frac{Z(t)}{Z'(t)}\frac{Z''(t)}{Z'(t)} = \frac{\omega-1}{\omega}$.  
When $\omega=1$ this second limit 
is equal to $0=\frac{\omega-1}{\omega}$. 

Hence  for $0<|t-\xi|<\delta$ we have $Z(t)$, $Z'(t)$ and $Z''(t)\ne0$ and 
by \eqref{eqkappaprime} we have
\begin{displaymath}
\lim_{t\to\xi}\kappa'(t)=\lim_{t\to\xi}\frac{1}{\pi}\frac{\frac{Z}{Z'}\vartheta''+\vartheta'-
\frac{Z}{Z'}\frac{Z''}{Z'}\vartheta'}{1+\left(\frac{Z}{Z'}\vartheta'\right)^2}=
\frac{1}{\pi}\Bigl(\vartheta'(\xi)-\frac{\omega-1}{\omega}\vartheta'(\xi)\Bigr)=
\frac{1}{\pi\omega}\vartheta'(\xi).
\end{displaymath}
\end{proof}

Assuming the RH and the simplicity of zeros we have 
\begin{displaymath}
\int_{a_\vartheta}^{\gamma_n}\kappa'(t)\,dt=n.
\end{displaymath}
Hence the mean value of $\kappa'(t)$ in $[0,t]$ is $\frac{N(t)}{t}$
which is approximately equal to
$\vartheta'(t)/\pi$. The above Proposition says that, assuming only the
simplicity of zeros,  at the points $\xi_n$  the value $\kappa'(\xi_n)$ is 
just equal to this density. 

\begin{figure}[H]
  \includegraphics[width=\hsize]{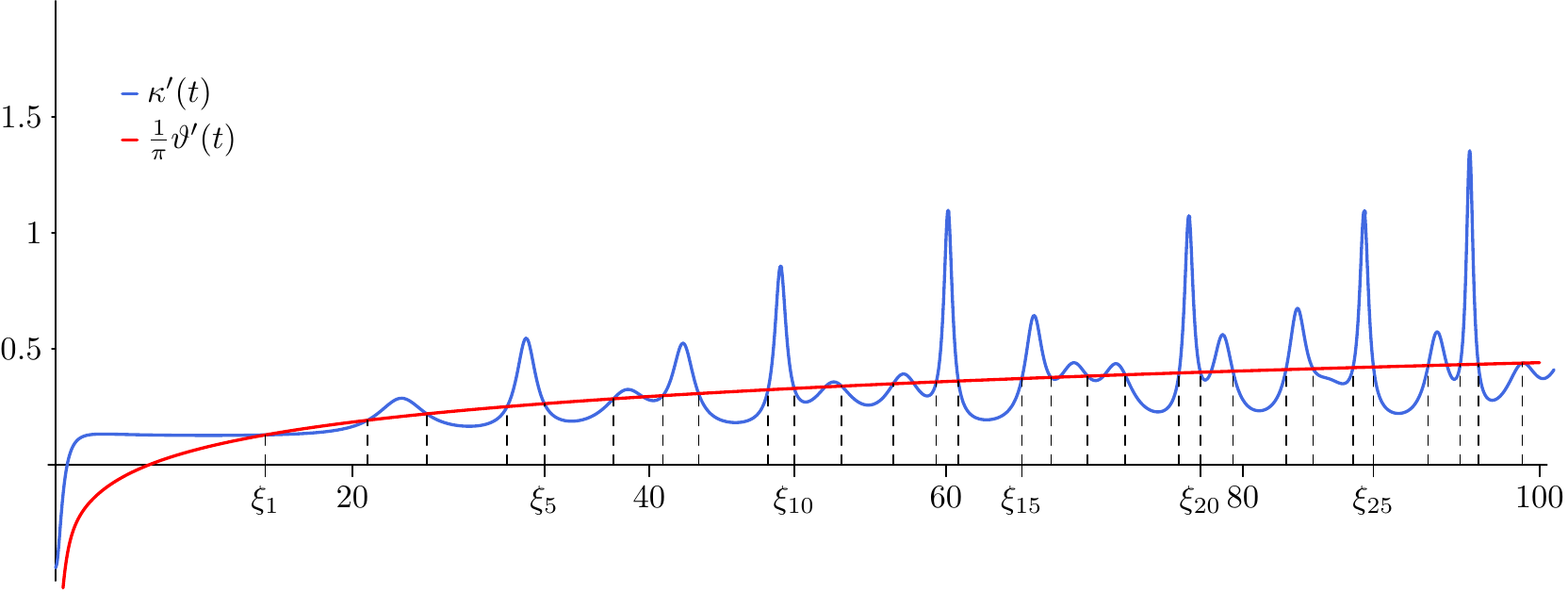}
  \caption{$\kappa(t)$}
  \label{F:two ways}
\end{figure}

Figure \ref{F:two ways} illustrates two ways in which the zeros of $\zeta'(s)$ determine the $\gg_n$
(assuming only simplicity of the zeros of zeta).   First $\xi_n$ is determined from 
$\kappa'(t)$  by the equation 
\begin{equation}\label{E:45}
\int_{\xi_0}^{\xi_n}\kappa'(t)\,dt=n\quad \text{or}\quad \int_{\xi_{n-1}}^{\xi_n}\kappa'(t)\,dt=1.
\end{equation}
Second, the points $\xi_n$ are intersections of the two curves $\kappa'(t)$ and 
$\vartheta'(t)/\pi$. Although, as we see in figure \ref{F:two ways} not all these intersections correspond
to points $\xi_n$. 

We can see how two close $\xi_n$ correspond to a peak in the graph of $\kappa'(t)$
that, according to equation \eqref{E:kappaprime} will be produced by one or more zeros 
$\beta'+i\gamma'$ of $\zeta'(s)$ with a relatively small $\beta'-\frac12$. 
Observe that equation \eqref{E:kappaprime} shows that $\kappa'(t)$ is fully determined by
the zeros of $\zeta'(s)$. 

Following these ideas we may improve (but assuming the RH) a theorem due to 
M. Z. Garaev, C. Y. Y{\i}ld{\i}r{\i}m \cite{GY}. 
For any given zero $\rho'=\beta'+i\gamma'$ of $\zeta'(s)$  let $\gamma_c$ be 
of all ordinates of zeros of
$\zeta(s)$, the one for which $|\gamma_c-\gamma'|$ is smallest (if there
are more than one such zero of $\zeta(s)$, take $\gamma_c$ to be the imaginary part
of any one of them).  Garaev and Y{\i}ld{\i}r{\i}m prove unconditionally
that $|\gamma_c-\gamma'|\ll |\beta'-\medio|^{\frac12}$.

\begin{proposition}[RH]\label{T:GY}
Assuming the RH, we have for any zero 
$\beta'+i\gamma'$ of $\zeta'(s)$ 
\begin{displaymath}
|\gamma_c-\gamma'|\le 1.9 \;|\beta'-\medio|^{\frac12}.
\end{displaymath}
\end{proposition}

\begin{proof}
Assuming the RH, $\beta'>1/2$ so that 
by equation \eqref{E:kappaprime}  we will have
\begin{displaymath}
\kappa'(t)\ge\frac{1}{2\pi}\log 2+\frac{f(t)}{\pi}+\frac{1}{\pi}\frac{\beta'-\frac12}{(\beta'-\frac12)^2+(t-\gamma')^2}.
\end{displaymath}
We will find an $a>0$ so that 
\begin{displaymath}
\kappa(\gamma'+a)-\kappa(\gamma'-a)\ge \frac{a\log2}{\pi}+\frac{2}{\pi}\arctan\frac{a}{\beta'-1/2}+\frac{1}{\pi}\int_{\gamma'-a}^{\gamma'+a}f(t)\,dt>1.
\end{displaymath}
Then there will be a point
$\xi\in[\gamma'-a,\gamma'+a]$ such that $\kappa(\xi)\in\Z$.  Then by Corollary \ref{integervalues}
$\zeta(\frac12+i\xi)=0$,  so that the ordinate $\gamma_c$ of the nearest zero of $\zeta(s)$
will satisfy  $|\gamma_c-\gamma'|\le a$. 

By \eqref{cota}, for $t>20$ we have $|f(t)|<41/40t$ and therefore for $a/\gamma'<1/2$
\begin{displaymath}
\Bigl|\frac{1}{\pi}\int_{\gamma'-a}^{\gamma'+a}f(t)\,dt\Bigr|\le 
\frac{41}{40\pi}\log\frac{\gamma'+a}{\gamma'-a}\le \frac{41}{40\pi}\frac{8}{3}\frac{a}{\gamma'}\le  \frac{a}{\gamma'},
\end{displaymath}
because $\log\frac{1+x}{1-x}\le 8x/3$ for $|x|\le 1/2$.
 
Therefore we want to choose $a$ such that 
\begin{displaymath}
\kappa(\gamma'+a)-\kappa(\gamma'-a)\ge \frac{a\log2}{\pi}+\frac{2}{\pi}\arctan\frac{a}{\beta'-1/2}-\frac{a}{\gamma'}>1
\end{displaymath}
or
\begin{displaymath}
a\frac{\log2}{2}-\frac{\pi a}{2\gamma'}\ge\frac{\pi}{2}-\arctan\frac{a}{\beta'-1/2}=
\arctan\frac{\beta'-1/2}{a} .
\end{displaymath}
It suffices to take 
\begin{displaymath}
a\Bigl(\frac{\log2}{2}-\frac{\pi}{2\gamma'}\Bigr)\ge \frac{\beta'-1/2}{a}
\end{displaymath}
Since $\gamma'\ge23$ 
it suffices to take
\begin{displaymath}
a=1.9\sqrt{\beta'-\medio}\ge \Bigl(\frac{\log2}{2}-\frac{\pi}{2\gamma'}\Bigr)^{-1/2}
\sqrt{\beta'-\medio}.
\end{displaymath}
Since always $\beta'<3$ and $\gamma'>23$ this $a$ satisfies  $a/\gamma'<1/2$, 
as used above.
\end{proof}

\section{The functions $E(t)$ and $S(t)$.}\label{S:E-S}

In the theory of the zeta function we consider the function 
\[S(t)=\pi^{-1}\arg\zeta(\medio+it)\] where the argument is obtained by its
continuous variation along the straight lines joining $2$, $2+it$, $\frac12+it$ starting with 
the value $0$. If $t$ is the ordinate of a zero, $S(t)$ is taken equal to $S(t+0)$. 
This function satisfies (see Edwards \cite{E}*{p.~173})
\begin{equation}\label{SN}
S(t)=N(t)-1-\frac{1}{\pi}\vartheta(t).
\end{equation}
If we assume the RH and the simplicity of the zeros, we will have 
$N(t)=N_{00}(t)=\lfloor\kappa(t)\rfloor$ (see Proposition \ref{P:kn0}).

We introduce 
a real analytic  version of $S(t)$ that we will call $E(t)$
\begin{equation}
E(t):= \pi+2\vartheta(t)+\ph\zeta'(\medio+it).
\end{equation}
By \eqref{E:kappa1} this is equivalent to 
\begin{equation}\label{Ekappa}
E(t)=3\frac{\pi}{2}+\vartheta(t)-\pi\kappa(t)
\end{equation}
with $E(0)=2\pi$. 

If $\frac12+i\gg_n$ is a simple zero of $\zeta(s)$ we will have
$E'(\xi_n)=0$ by Proposition \ref{kappaprimeatzeros}. The converse is not true. 
For example at $t_0=39.587\,127\,340\dots $ the function $E(t)$ has a local 
minimum with $E(t_0)=0.151\,790\,437\dots$ It is also easy to show
that $E(t)-2\pi$ is a real analytic  odd function.

In fact $E'(t)=\vartheta'(t)-\pi\kappa'(t)$ so that the zeros of $E'(t)$ are just
the points where the graphs of $\frac{1}{\pi}\vartheta'(t)$ and $\kappa'(t)$ 
intersect (see Figure \ref{F:two ways}). By equation \eqref{eqkappaprime}  
for $Z'(t)^2+(Z(t)\vartheta'(t))^2\ne0$ we have
\begin{equation}
E'=\vartheta'-\pi\kappa'=Z\cdot\frac{Z\vartheta'^3-Z'\vartheta''+Z''\vartheta'}
{(Z')^2+(Z\vartheta')^2}.
\end{equation}

\begin{figure}[H]
  \includegraphics[width=\hsize]{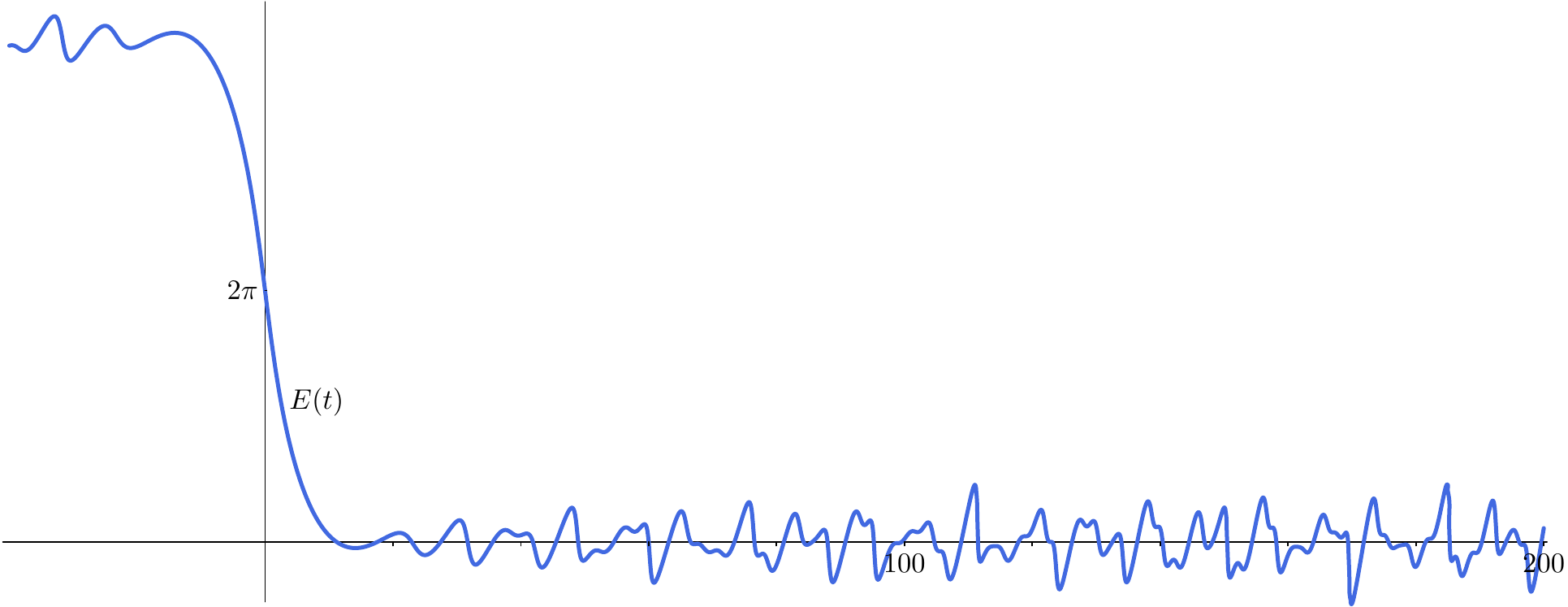}
  \caption{$E(t)$}
  \label{F:F44}
\end{figure}

For the next Proposition we need a measure of the possible failure of the RH.  
\begin{definition}
For any $t>0$ we define $\RH(t)$ by
\begin{equation}
\RH(t):= N(t)-N_{00}(t).
\end{equation}
\end{definition}

That is $\RH(t)$  is equal to the number of zeros $\beta+i\gamma$ of $\zeta(s)$ with 
$0<\gamma\le t$ and $\beta\ne\frac12$, plus the number of zeros $\beta'+i\gamma'$ of 
$\zeta'(s)$  with $\beta'=\frac12$ and $0<\gamma'\le t$ all of them counted with 
their multiplicities. By Proposition \ref{P:doblezeros}
these zeros of $\zeta'(s)$ will be multiple zeros of $\zeta(s)$ on the critical line. 
We have $\RH(t)=0$ if and only if the zeros $\beta+i\gamma$ of $\zeta(s)$ with 
$0<\gamma\le t$ are all on the critical line and are simple.

\begin{proposition}
We have 
\begin{equation}
-\medio+\RH(t)<S(t)+\fracpi E(t)\le \medio+\RH(t),\qquad   t>a_\vartheta.
\end{equation}
\end{proposition}

\begin{proof}
By \eqref{SN} and \eqref{Ekappa} we have
\begin{align*}
S(t)+\fracpi E(t)&=N(t)-1-\fracpi\vartheta(t)+\fracpi\left(\terpi+\vartheta(t)-\pi\kappa(t)\right)\\
&=N(t)+\medio-\kappa(t)=\RH(t)+\medio-\kappa(t)+N_{00}(t)
\end{align*}
so that by \eqref{Nkappa} for $t>a_\vartheta$ we have
\begin{equation}\label{SpluspiE}
S(t)+\fracpi E(t)=\RH(t)-\left(\kappa(t)-\lfloor\kappa(t)\rfloor-\medio\right),\qquad  t>a_\vartheta
\end{equation}
from which the result follows.
\end{proof}

\begin{corollary}
Assuming the RH and the simplicity of the zeros we will have
\begin{equation}
-\medio<S(t)+\fracpi E(t)\le \medio,\qquad   t>a_\vartheta.
\end{equation}
\end{corollary}

Indeed, the hypotheses are equivalent to $\RH(t)=0$. 
By the well known Fourier series of $\widetilde{B}_1(x)=x-\lfloor x\rfloor-\frac12$ we get from \eqref{SpluspiE},
under the assumptions of the Corollary
\begin{equation}\label{SpluspiE2}
S(t)+\fracpi E(t)=2\sum_{n=1}^\infty \frac{\sin(2\pi n\kappa(t))}{2\pi n},\qquad  t>a_\vartheta.
\end{equation}

\section{Extension to other $L$-functions.}
Most of the formulas and functions defined in this paper for $\zeta(s)$ can be 
generalized to other functions, including the Selberg class. The main thing we need is
a functional equation. So let's assume that we have a Dirichlet series
\begin{displaymath}
f(s)=\sum_{n=1}^\infty \frac{a_n}{n^s}
\end{displaymath}
\begin{figure}[H]
  \includegraphics[width=\hsize]{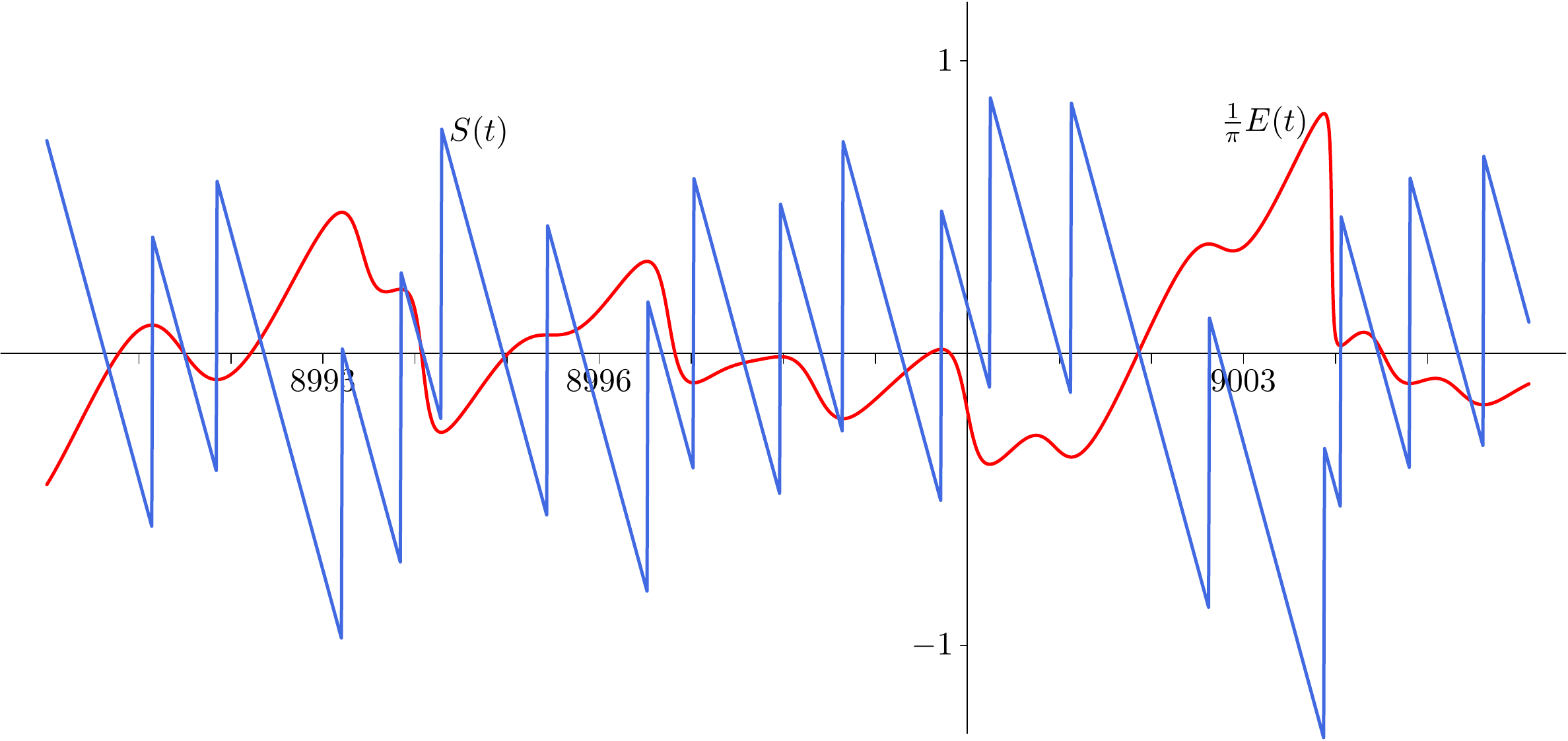}
  \includegraphics[width=\hsize]{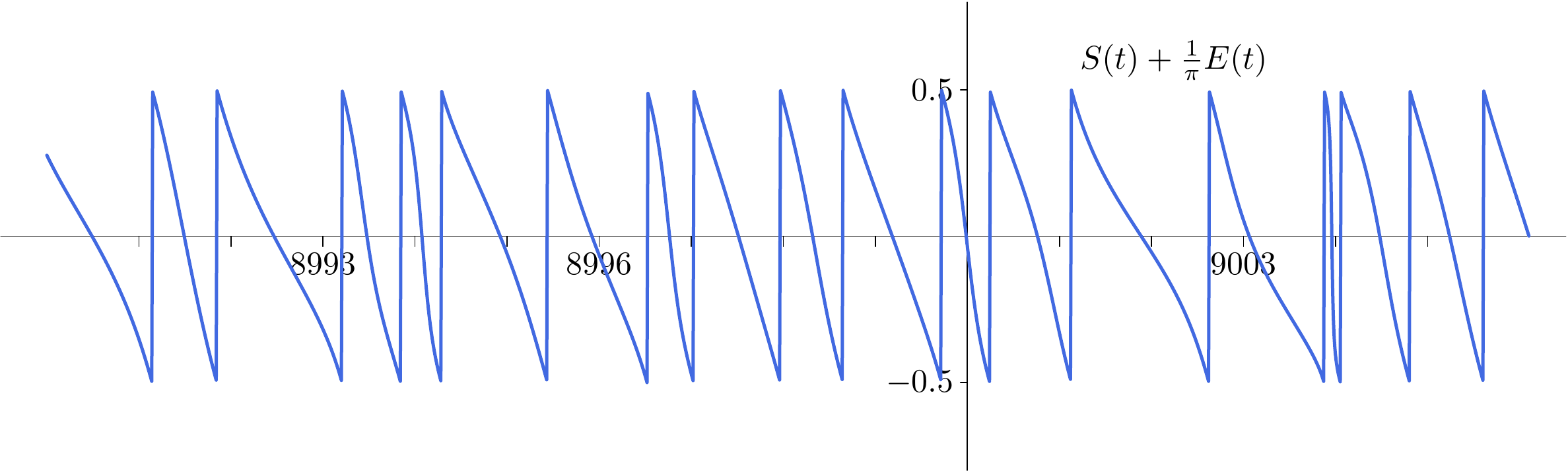}
  \caption{Plots of $S(t)$, $\frac{1}{\pi}E(t)$ and  $S(t)+\frac{1}{\pi}E(t)$ for $t$ in $(8990,9006)$.}
  \label{F:SversusE}
\end{figure}

\noindent which can be extended as a meromorphic function to the plane $\C$, in such a way 
that there exist
numbers $Q>0$, $\alpha_n>0$ and $r_n\in\C$ with $\Re(r_n)\ge0$ such that
\begin{displaymath}
\Phi(s):=Q^sf(s)\prod_{n=1}^d \Gamma(\alpha_n s+r_n)
\quad\text{satisfies}\quad  \Phi(s)=w\overline{\Phi(1-\overline{s})}
\end{displaymath}
where $w$ is a complex number of modulus $|w|=1$.
In this way all Dirichlet series for a primitive character, and the Dirichlet series $f(s)$
considered by Titchmarsh \cite{T}*{Section 10.25}, which has no Euler product, and 
does not satisfy an RH will be included. 

Putting $s=\medio+it$ the functional equation leads to 
\begin{displaymath}
\frac{f(\medio+it)}{\overline{f(\frac12+it)}}=
wQ^{-2it}\prod_{n=1}^d\frac{\overline{\Gamma(\alpha_n(\frac12+it)+r_n)}}
{\Gamma(\alpha_n(\frac12+it)+r_n)}
\end{displaymath}

Therefore, if we define
\begin{equation}
\vartheta(f,t):=-\frac{\arg w}{2}+t\log Q+\sum_{n=1}^d\ph 
\Gamma(\alpha_n(\medio+it)+r_n)
\end{equation}
this will be a real analytic function and  $\ph f(\frac12+it)=-\vartheta(f,t)$
so that 
\begin{equation}
f(\medio+it)=e^{-i\vartheta(f,t)} Z(f,t)
\end{equation}
where $Z(f,t)$ is a real valued  real analytic function of the real variable $t$.
It is not  difficult to define functions $\kappa(f,t)$, $E(f,t)$, and so on.

\begin{figure}[H]
  \includegraphics[width=\hsize]{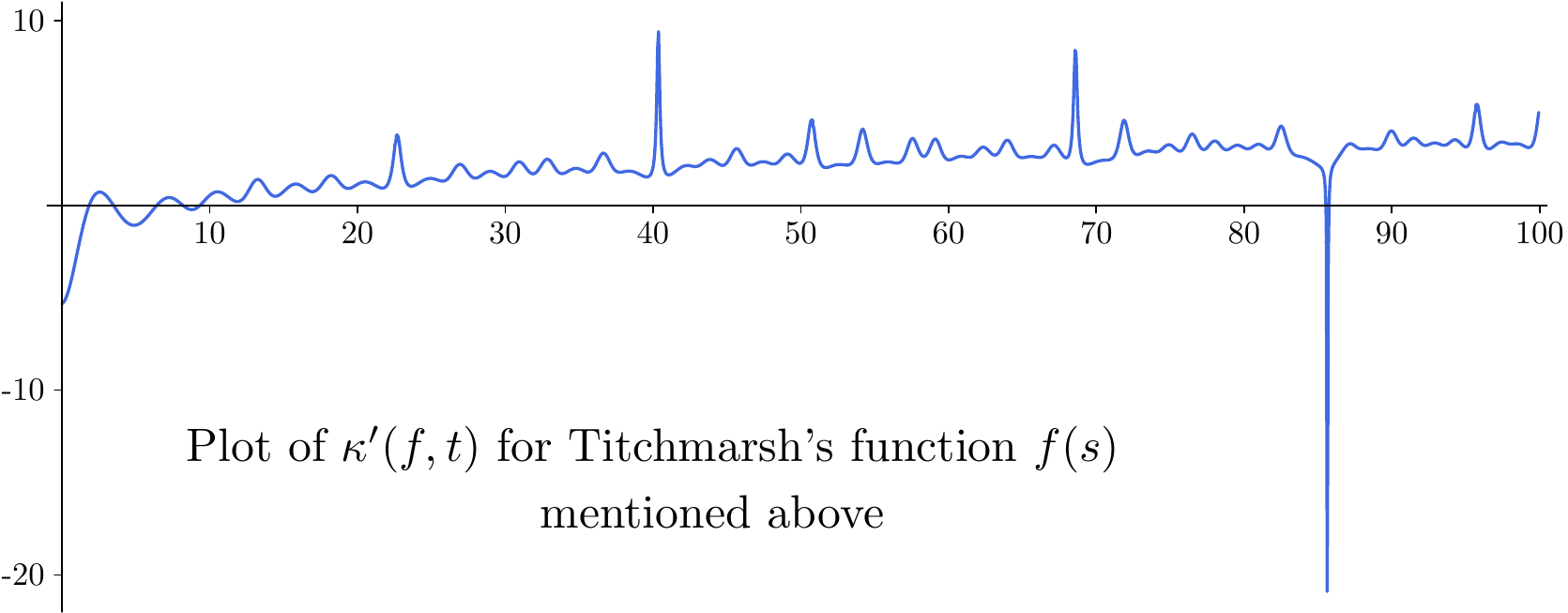}
  \caption{This  Dirichlet series has a zero  at the point 
  $\rho\approx0.80851 718  + i\; 85.69934 848$.}
  \label{F:titchmarshfunction}
\end{figure}

\noindent\textsc{Acknowledgement: } The authors would like to thank
Patrick R. Gardner ( Kennewick, Washington, USA )  for his linguistic 
assistance in preparing this note,
and for his interest in the subject.

% bibliography

\end{document}